\documentclass[journal]{IEEEtran}
\ifCLASSINFOpdf

\else
\fi

\usepackage{graphicx}
\usepackage{setspace}
\usepackage{amsmath}
\usepackage{amsthm}
\usepackage{color}
\usepackage{amssymb}
\usepackage{setspace}
\usepackage[noend]{algpseudocode}
\usepackage{algorithmicx}
\usepackage{algorithm}
\usepackage{pifont}
\usepackage{subfigure}
\usepackage{cite}
\usepackage{epstopdf}
\usepackage{bm}
\usepackage{multirow}

\newtheorem{assum}{Assumption}
\newtheorem{prop}{Proposition}
\newtheorem{rem}{Remark}
\newtheorem{lem}{Lemma}
\newtheorem{thm}{Theorem}
\newtheorem{cor}{Corollary}

\begin{document}
\title{Disturbance-resilient Distributed Resource Allocation over Stochastic Networks using Uncoordinated Stepsizes}

\author{Tie~Ding, Shanying~Zhu, \IEEEmembership{IEEE Member}, Cailian~Chen, \IEEEmembership{IEEE Member}, Xinping~Guan, \IEEEmembership{IEEE Fellow}% <-this % stops a space
\thanks{The authors are with the Department of Automation, Shanghai Jiao Tong University, Shanghai 200240, China, and also with Key Laboratory of System Control and Information Processing, Ministry of Education of China, Shanghai 200240, China.

This work is supported in part by National Key Research and Development
Program of China (No. 2016YFB0901900), NSF of China under the grants
61922058, 61633017, and 61521063, and NSF of Shanghai Municipality of
China under grant 18ZR1419900.
}}% <-this % stops a space
%\thanks{J. Doe and J. Doe are with Anonymous University.}% <-this % stops a space
%\thanks{Manuscript received April 19, 2005; revised August 26, 2015.}}

%\markboth{Journal of \LaTeX\ Class Files,~Vol.~14, No.~8, August~2015}%
%{Shell \MakeLowercase{\textit{et al.}}: Bare Demo of IEEEtran.cls for IEEE Journals}
% The only time the second header will appear is for the odd numbered pages
% after the title page when using the twoside option.
% 
% *** Note that you probably will NOT want to include the author's ***
% *** name in the headers of peer review papers.                   ***
% You can use \ifCLASSOPTIONpeerreview for conditional compilation here if
% you desire.

% If you want to put a publisher's ID mark on the page you can do it like
% this:
%\IEEEpubid{0000--0000/00\$00.00~\copyright~2015 IEEE}
% Remember, if you use this you must call \IEEEpubidadjcol in the second
% column for its text to clear the IEEEpubid mark.

% use for special paper notices
%\IEEEspecialpapernotice{(Invited Paper)}

% make the title area
\maketitle

% As a general rule, do not put math, special symbols or citations
% in the abstract or keywords.
\begin{abstract}
This paper studies distributed resource allocation problem in multi-agent systems, where all the agents cooperatively minimize the sum of their cost functions with global resource constraints over stochastic communication networks. This problem arises from many practical domains such as economic dispatch in smart grid, task assignment and power allocation in robotic control. Most of existing works cannot converge to the optimal solution if states deviate from feasible region due to disturbance caused by environmental noise, misoperation, malicious attack, etc. To solve this problem, we propose a distributed deviation-tracking resource allocation algorithm and prove that it linearly converges to the optimal solution with constant stepsizes. We further explore its resilience properties of the proposed algorithm.
Most importantly, the algorithm still converges to the optimal solution under the disturbance injection and random communication failure. In order to improve the convergence rate, the optimal stepsizes for the fastest convergence rate are established. We also prove the algorithm converges linearly to the optimal solution in mean square even with uncoordinated stepsizes, i.e.,  agents are allowed to employ different stepsizes. Simulations are provided to verify the theoretical results.

%To solve this problem, we propose an initialization-free distributed optimal resource allocation algorithm via state tracking and prove that it linearly converges to the optimal solution even with constant stepsizes. This is different from most existing works that adopt vanishing stepsizes and cannot ensure linear convergence rate. Moreover, in order to improve the convergence rate, the optimal stepsize for fastest convergence rate is established. Simulations are provided to verify the theoretical results. 
\end{abstract}

% Note that keywords are not normally used for peerreview papers.

% For peer review papers, you can put extra information on the cover
% page as needed:
% \ifCLASSOPTIONpeerreview
% \begin{center} \bfseries EDICS Category: 3-BBND \end{center}
% \fi
%
% For peerreview papers, this IEEEtran command inserts a page break and
% creates the second title. It will be ignored for other modes.
\IEEEpeerreviewmaketitle

\section{Introduction}
\IEEEPARstart{R}ecently, distributed resource allocation problem has attracted much attention due to its wide application in various practical problems, such as economic dispatch in energy network \cite{52,53,54}, channel allocation in wireless communication \cite{55,56} and computing resource assignment in edge computing \cite{57}. In  resource allocation problems, 
%agents cooperate to achieve an agreement on the assignment of limited resource for optimal profit with only local information and communications with neighbors. Traditional 
centralized methods require an entity to collect information of agents and distribute the strategies of resource allocation or task assignment to all the agents \cite{45}. 
%However, centralized methods cannot perform well in large-scale networks due to 
It suffers from issues of 
high requirement of synchronization, heavy cost of long-distance communication and poor scalability. 
%In centralized network, the whole system will stop working if the center fails while the distributed network is robuster to possible communication failure. Moreover, 
Distributed resource allocation approaches avoid long-distance communication and make the network more scalable \cite{37}. Since each node only has local knowledge, it requires a reliable communication network to achieve global optimization. Due to the vulnerability nature of wireless communications, the network is easily affected by potential attack and environmental noise. This may lead to a stochastic communication network suffering from random failure, which results in information island, inaccurate estimation of the optimal solution and, eventually, inexact and stochastic convergence result. 
Therefore, it is significant to design proper distributed algorithms to obtain the optimal strategies effectively and simultaneously alleviate the stochasticity of convergence result resulting from link failures.
%is of great value.
%However, 
%since each node only has local knowledge and requires 
%each node only has local knowledge and the commu, 
%how to design proper distributed algorithms to obtain the optimal strategies effectively is of great value.

Distributed resource allocation problems are actually the dual problems of distributed consensus optimization, which requires all agents' states to be equal while the optimal point of distributed resource allocation problems is obtained when achieving consensus on marginal costs. Due to this property, \cite{46} proposes decentralized resource allocation algorithm that adopts weighted gradient method to ensure the consensus on gradients. Extension of the result \cite{46} to random gradients, the authors \cite{47} design a random coordinate descent algorithm based on weighted gradient method and prove that it converges to the optimal solution linearly in probability. 
However, both \cite{46} and \cite{47} are only suitable for fixed communication network and does not consider stochastic network with communication failure. 
The authors of \cite{67} extend the algorithm of \cite{46} to time-varying networks and prove that it converges to the optimal solution in quadratic time. 
But \cite{67}, as well as \cite{46,47}, requires the states to be kept feasible through all iterations.
If the state is disturbed by noise or malicious attacks at any moment, which causes infeasible states, a derivation from the optimal point occurs inevitably.

%require the states to be kept feasible all through the iteration.
%If the state is disturbed by noise or malicious attacks at any moment, which causes infeasible states, resulting in a derivation from the optimal point.
%Moreover, \cite{46} and \cite{47} are only suitable for fixed communication network and does not consider stochastic network with communication failure.
%then the algorithm proposed in \cite{46} and \cite{47} cannot resile and the convergence point is also infeasible.

To deal with infeasible states, \cite{48} proposes an initialization-free distributed algorithm to solve optimal resource allocation problems with local constraints. 
%It is proved to converge to the optimal solution exponentially. 
But it is a continuous-time algorithm requiring infinite bandwidth to support the data exchange between agents and 
the convergence rate is far from optimal. Additionally, continuous-time algorithms cannot deal with stochastic communication network because it is difficult to define derivative with stochastic variables. 
% and dose not provide the methods to improve convergence rate. Recently, discrete-time distributed resource allocation attracts a lot attention.  
Ref.\cite{49} proposes a distributed algorithm over dynamic networks, considering the uncertainty of local parameters. It converges to the optimal point but a decaying stepsize is needed, which results in a slower convergence rate.
%of $\mathcal{O}(\ln(k)/\sqrt{k})$. 
A decaying stepsize is also adopted in \cite{50}. Using constant stepsize, \cite{36} proposes a distributed resource allocation with dual splitting approach (DuSPA), which ensures that both primary and dual states converge to the optimal point. But it can only be used in fixed network. Extending primary-dual methods to time-varying networks, the authors of \cite{68,69} propose distributed resource allocation algorithms which converge linearly to the optimal solution with constant stepsizes. Both \cite{68} and \cite{69} assume that the network is jointly strongly connected, i.e., each communication link should have bounded intercommunication interval. For stochastic networks, however, we cannot determine the bound of communication interval, especially when the possibility of packet loss is high. Moreover, these aforementioned works all adopt identical  stepsizes for all agents.
It is difficult for all agents to realize the consensus of stepsizes in distributed manner due to random communication failure. 
%Furthermore, the aforementioned works does not consider the random failure of communication link.

As one typical application of distributed resource allocation, economic dispatch problems in smart grid, where a group of microgrids cooperatively minimize the cost of generation subject to the balance between supply and demand, also gain lots of attention. The authors of \cite{61} propos an incremental cost consensus algorithm that solves economic dispatch problem in a distributed way. However, it requires a center entity to maintain the balance between supply and demand. To eliminate the centralized node,  a consensus + innovation approach is proposed in \cite{62} to deal with distributed energy management. Although it can be operated in a completely distributed way, a decaying stepsize is used in this algorithm, leading to slow convergence rate. Ref. \cite{63} proposes consensus based method which uses an auxiliary variable to estimate the mismatch between supply and demand. Convergence is proved when stepsize is small enough but no specific upper bound is provided. 
It is worth noting that \cite{61}-\cite{63} and many other related works are only suitable for those strictly quadratic cost functions of generation, which limits scope of application.
Moreover, these aforementioned works do not consider the communication failure, which may be caused by limited transmission energy, environmental noise, malicious attack, etc. Communication failure may cause the power system to be operated in a non-optimal condition, which will greatly increase the productivity cost \cite{65}.

%Considering the communication failure in the distributed power, \cite{64} designs a consensus based approach by using frequency controller. However, this algorithm only deals with the consensus of Lagrangian multiplier but does not prove the convergence of state to the optimal solution.

\iffalse
For example, in economic dispatch problems of multiple-energy network, constrained by the balance of supply and demand, all the micro energy grid need to cooperate to decide the generation strategies to minimize the total generation cost while maintaining the balance between supply and demand. However, it should be noted that most of existing works about resource allocation do not consider the transformation between different types of resource, which, however, exists in many practical problems. For instance, in multiple-energy network, redundant natural gas can be used to generate electricity by gas turbine to make up for insufficient electricity supply, which means we cannot independently program the allocation of all the types of resources. A recent work \cite{64} considers the communication failure in power systems. However, this algorithm only deals with the consensus of Lagrangian multipliers but does not prove the convergence to the optimal solution. 
\fi

%we propose a distributed algorithm that minimize the sum of all agents cost functions while keep the balance between supplies and demands of multiple coupled resources. 
%The contributions of this paper are summarized as follows:
We aim to solve distributed resource allocation problem over stochastic networks where communication links randomly fail and the states are injected by random disturbance. 
%A recent work \cite{64} also considers the communication failure in power systems. However, this algorithm only deals with the consensus of Lagrangian multipliers but does not prove the convergence to the optimal solution. 
In this paper, 
a disturbance-resilient distributed algorithm targeting the communication failures is proposed with guaranteed linear convergence to the optimal solution. 

%The algorithm overcomes the drawback of most existing works that states cannot converge to the optimal solution under disturbance injected to states. 
The contributions of this paper are shown as follows.
\begin{enumerate}
\item
%relax the requirement of feasible initialization, which most of existing literatures 
%for cost functions with Lipschitz gradients and strong convexity. 
%This algorithm is designed based on consensus theory and state tracking scheme, which enables the algorithm to be implemented in a completely distributed manner. 
We propose a disturbance-resilient distributed algorithm to solve distributed resource allocation problems over stochastic networks with random communication failure and stochastic disturbance on states. 
This algorithm is a combination of the deviation-tracking technique and weighted gradient scheme.
Different from most economic dispatch algorithms that only are suitable for quadratic cost functions \cite{61}-\cite{63}, the proposed algorithm works for general functions with Lipschitz gradients and strong convexity. Moreover, 
compared with \cite{68, 69},
we relax the assumption of 
joint connectivity of networks over certain bounded intercommunication interval that used in  to connectivity in mean. 
%bounded intercommunication interval which guarantees the network is jointly connected. We only require the network to be connected in mean.
\item
We prove that the proposed algorithm converges linearly to the optimal solution in mean square even with random communication failure. A specific upper bound of the stepsizes that ensure convergence is provided. 
We further explore the resilience properties of the proposed algorithm. Compared with the algorithms proposed in \cite{46,47,67,36}, the proposed algorithm converges in mean square to the optimal solution even with disturbance on states while these works cannot resile from disturbance. We provide comparative simulations in Section VII B.
%Different from traditional weighted gradient method which cannot converge to the optimal solution even if one agent's state is disturbed at only one moment, the proposed algorithm converges in mean square to the optimal solution even with disturbance on states. 
%We propose a resource allocation algorithm that works for general functions with Lipschitz gradients and strongly convexity. This is different from most economic dispatch algorithms that only are suitable for quadratic cost functions.
%This algorithm is a combination of deviation-tracking method and weighted gradient scheme.
%The proposed algorithm is proved to converge linearly to the optimal solution in mean square over stochastic network with random communication failure. We provide specific upper bound of the stepsizes that ensure convergence.
%To show the significance of resiling from disturbance, 
%Different from traditional weighted gradient method which cannot converge to the optimal solution even if one agent's state is disturbed at only one moment, the proposed algorithm converges to the optimal solution even there exists disturbance on states. We prove that even if the disturbance is unbounded, the proposed algorithm still converges in mean square as long as the disturbance added to the states vanishes to 0 in mean square. 
\item
Based on the convergence results, we obtain the estimate of the convergence rate.
To improve the convergence rate, the optimal stepsizes within the upper bound are established. 
Furthermore, we prove that the algorithm converges to the optimal solution in mean square even with uncoordinated stepsizes.

%\item
%To improve the convergence rate, we establish the optimal stepsizes within the upper bound based on the convergence reasult over fixed communication network.
%Furthermore, considering the stochastic network with random communication failure, the algorithm still converges to the optimal solution in mean square at a linear rate. 
%Additionally, the proposed algorithm does not require the initial state to be feasible and it can be chosen as arbitrary value. 
%We provide specific upper bound of the stepsizes that ensure convergence. Based on this result, we establish the optimal stepsizes within the upper bound for improving the convergence rate.
\end{enumerate}

\emph{Notation}: We denote the set of $n$-dimensional vectors and $n\times u$-dimensional matrices by $\mathbb{R}^n$ and $\mathbb{R}^{n\times u}$, respectively. $\bm{0}, \bm{1}\in\mathbb{R}^n$ represents the vectors of zeros and ones, respectively. $\bm{I}\in\mathbb{R}^{n\times n}$ is the identity matrix. $\Vert\cdot\Vert_F$ denotes the Frobenius norm for matrixes and $\Vert\cdot\Vert_2$ represents the Euclidean norm for vectors. $\mathbb{E}\{\cdot\}$ means the expectation and 
for any vector $\bm{v}\in\mathbb{R}^{n}$, we define expectation norm $\Vert\bm{v}\Vert_E\triangleq\sqrt{\mathbb{E}\{\Vert\bm{v}\Vert_2^2\}}$. For any vector sequence $\{\bm{v}(k)\}_{k=1}^{\infty}$, $\bm{v}(k)$ converges to $\bm{v}^*$ in mean square if $\lim_{k\rightarrow\infty}\Vert\bm{v}(k)-\bm{v}^*\Vert_E=0$.
%Let $Gau(\phi, \sigma^2)$ be the Gaussian distribution with PDF $f_G(x)=\frac{1}{\sigma \sqrt{2\pi}}e^{-\frac{(x-\phi)^2}{2\sigma^2}}$, where $\phi$ and $\sigma^2$ denote the mean and variance, respectively. For $n$-dimensional vector $\bm{x}=[x_1,...,x_n]^T$, $\bm{x}\sim Gau_n(\phi, \sigma^2)$ means that $x_i\sim Gau(\phi, \sigma^2), \forall i$.
%$x_i\sim Gau(\phi, \sigma^2), \forall i$, we define 
%$\lambda_{\max}(\cdot)$ and $\lambda_{\min}(\cdot)$ mean the maximal and minimal eigenvalues, respectively. 
$\rho(\cdot)$ denotes the spectral radius of matrixes. For any real symmetric matrix $\bm{A}\in\mathbb{R}^{n\times n}$, let $\{\lambda_i(\bm{A})\}_{i=1}^n$ be the eigenvalues such that $\lambda_n(\bm{A})\leq\lambda_{n-1}(\bm{A})\leq...\leq\lambda_1(\bm{A})$.
%\iffalse
%In this paper, 
We denote by $x_i$ the local copy of the global variable $x\in \mathbb{R}^u$ at agent $i$. Its value at time $k$ is denoted by $x_i(k)$. 
%$x_i(k), y_i(k), z_{\zeta i}(k), z_{\eta i}(k), \zeta_i(k)$ and $\eta_i(k)$, $i\in\{1,...,n\}$ are agent $i$'s local variables, which are all in  $\mathbb{R}^u$. 
%$x_{ij}(k),y_{ij}(k),\zeta_{ij}(k),\eta_{ij}(k)\in \mathbb{R}$ denote the $j$-th component of vectors $x_i(k), y_i(k), \zeta_i(k), \eta_i(k)$, $j\in\{1,...,u\}$.
%$\bm{x}(k), \bm{y}(k), \bm{z_{\bm{\zeta}}}(k), \bm{z_{\bm{\eta}}}(k), \bm{\zeta}(k)$, and $\bm{\eta}(k)$ are matrixes in $\mathbb{R}^{n\times u}$ and denote the stacked form of $x_i(k), y_i(k), z_{\zeta i}(k), z_{\eta i}(k), \zeta_i(k)$ and $\eta_i(k)$, $i\in\{1,...,n\}$. 
%Additionally, $x_{ij}(k),y_{ij}(k),\zeta_{ij}(k),\eta_{ij}(k)\in \mathbb{R}$ denote the $j$-th component of vectors $x_i(k), y_i(k), \zeta_i(k), \eta_i(k)$, $j\in\{1,...,u\}$. 
%For example, 
We introduce the stacked matrix $\bm{x}=[x_{1},...,x_{n}]^T\in\mathbb{R}^{n\times u}$
\iffalse
\begin{align}
\bm{x}=\left[
\begin{matrix}
x_{1}^T\\
\vdots\\
x_{n}^T\\
\end{matrix}
\right]=\left[
\begin{matrix}
x_{11}&\dots&x_{1u}\\
\vdots&&\vdots \\
x_{n1}&\dots&x_{nu}\\
\end{matrix}
\right]\in\mathbb{R}^{n\times u}
\notag
\end{align}
\fi
and define $\bar{\bm{x}}\triangleq\frac{\bm{11}^T}{n}\bm{x}$, $\check{\bm{x}}\triangleq\bm{L}\bm{x}$, where $\bm{L}=\bm{I}-\frac{\bm{11}^T}{n}$.

\section{Problem Formulation}
	Consider a distributed resource allocation problem in multi-agent systems. For any $i\in\{1,...,n\}$, agent $i$ has its individual cost function $f_i: \mathbb{R}^u\rightarrow\mathbb{R}$, which is only known to agent $i$ itself. All the agents cooperate to minimize the sum of their cost functions with specified amount of resource:
\begin{equation}
\begin{split}
\min_{\bm{x}\in\mathbb{R}^{n\times u}} f(\bm{x})=\sum\limits_{i=1}^nf_i(x_i)\ \  s.t.\sum_{i=1}^{n}{x}_i=\sum_{i=1}^{n}{d}_i,
\end{split}
\label{problem1}
\end{equation}
where ${d}_i, \forall i\in\{1,...,n\}$ denotes the local demand of resource, which is only known by agent $i$ and does not share with other agents for privacy concern.
Here we make the following standard assumptions about the functions $f_i, i\in\{1,...,n\}$.
\begin{assum}\label{lip}
For every $i$, $f_i: \mathbb R^u \rightarrow \mathbb R$ is differentiable and has Lipschitz gradient with $\eta_i>0$,
\begin{equation}
\Vert \nabla f_i(x_a)-\nabla f_i(x_b)\Vert _2\leq \varphi_i\Vert x_a-x_b\Vert _2, \forall x_a,x_b\in \mathbb R^u.
\end{equation}
%where $\eta_i$ is the Lipschitz constant for the function $f_i$.
\end{assum}
\begin{assum}\label{convex}
For every $i$, $f_i: \mathbb R^u \rightarrow \mathbb R$ is strongly convex, i.e., there exists $\varphi_i>0$ such that $\forall x_a,x_b\in \mathbb R^u$,
\begin{equation}
(x_a-x_b)^T\left(\nabla f_i(x_a)-\nabla f_i(x_b)\right)\geq\eta_i\Vert x_a-x_b\Vert _2^2. \label{conve}
\end{equation}
\end{assum}

\iffalse
Considering specified amount of resource, the resource allocation is constrained by 
\begin{align}
\sum_{i=1}^{n}\bm{x}_i=\sum_{i=1}^{n}\bm{d}_i
\label{problem1con}
\end{align}
where $\bm{d}_i$ denotes the local resource of agent $i$. 
\fi

%Problem \eqref{problem1} is a typical form in many resource allocation problems, one of which is the economic dispatch problem in power systems, where all the generators produce certain power for the total load of the networks, see \cite{59,58,51}. 

%\subsection{Stochastic communication network}
We model the topology of the network over which agents communicate with each other as a random graph $\mathcal{G}(k)=(\mathcal{V},\mathcal{E}(k))$, where $\mathcal{V}$ is the set of agents $\mathcal{V}=\{1,2,...,n\}$, $\mathcal{E}(k)\in\mathcal{V}\times\mathcal{V}$ denotes the edges. Let $\bm{W}(k)=[w_{ij}(k)]_{n \times n}\in\mathbb{R}^{n \times n}$, where $w_{ij}(k)$ is the weight of the edge $(i,j)$ at time $k$. 
$w_{ij}(k)>0$ means that agent $i$ can communicate with agent $j$ at time $k$. Since each agent has local knowledge, we have $w_{ii}(k)>0, \forall i, k$. 
All the agents in $\mathcal{N}_i(k)=\{j|w_{ij}(k)>0\}$ are called neighbors of agent $i$. 
Here, we consider that 
each communication link is subject to a random failure, that is, for any agent $i$ and $j\in\mathcal{N}_{i}(k)$, we have $\mathbb{P}\{w_{ij}(k)>0\}=\theta_{ij}(k)$ and $\mathbb{P}\{w_{ij}(k)=0\}=1-\theta_{ij}(k)$, $0<\theta_{ij}(k)<1$. Specifically, $\mathbb{P}\{w_{ii}(k)>0\}=1, \forall i$.
We assume that if link $(i, j)$ fails, link $(j, i)$ also fails.

Here, we provide a method to determine the weight $w_{ij}(k), \forall i,j,k$.
For each time $k$, agent $i$ sends information to agent $j$, along with the weight $w_i^j(k)$, $j\in\mathcal{N}_i$.  If the communication link between $i$ and $j$ works, agent $i$ and $j$ will receive $w_j^i(k)$ and $w_i^j(k)$, respectively. The weight of the edge $(i, j)$ is set as the smaller one between $w_i^j(k)$ and $w_j^i(k)$, i.e.,
\begin{align}
w_{ij}(k)=
\begin{cases}
\min\{w_i^j(k), w_j^i(k)\} & \text{if link}\  (i, j)\  \text{works}
\\0& \text{if link}\  (i, j)\ \text{fails}
\end{cases}.
\end{align}
Then, for each agent $i$, 
\begin{align}
w_{ii}(k)=1-\sum_{j\in\mathcal{N}_i}w_{ij}(k),
\end{align}
and thus $\sum_{j=1}^{n} w_{ij}(k)=1$ and $\sum_{j=1}^{n}w_{ji}(k)=1, \forall i, k$.

Based on the above method to obtain the weights, we make the following assumption on the communication network \cite{16, 66}. 

\begin{assum}\label{network}
The weight matrices $\{\bm{W}(k)\}_{k=0}^{\infty}$ are a sequence of i.i.d matrices from some probability space $\mathcal{F}=(\Omega, \mathcal{B}, \mathcal{P})$ such that each $\bm{W}(k)$ is symmetric, doubly stochastic, i.e., $\forall k$
\begin{align}
\bm{W}(k)=&\bm{W}(k)^T,\ \  \bm{1}^T\bm{W}(k)=\bm{1}^T,\ \ \bm{W}(k)\bm{1}=\bm{1},\label{doublestochastic}
\end{align}
and 
%\\&\bm{\mu}(k)=\rho\left(\mathbb{E}\{(\bm{W}(k))^T\bm{W}(k)\}-\frac{\bm{11}^T}{n}\right)<1.\label{stonet}
\begin{align}
\rho\left(\mathbb{E}\{\bm{W}(k)\}-\frac{\bm{11}^T}{n}\right)<1.\label{con1}
%\\&\lambda_{2}\left(\mathbb{E}\{(\bm{I}-\bm{W}(k))^T(\bm{I}-\bm{W}(k))\}\right)>0\label{con2}
\end{align}
%$\mathcal{G}$ is connected and the matrix $\bm{W}\in\mathbb{R}^{n\times n}$ is doubly stochastic, i.e., $\bm{W1}=\bm{1}$ and  $\bm{1}^T\bm{W}=\bm{1}^T$.
\end{assum}
%Let $\{\lambda_i(\bm{W}^T\bm{W})\}_{i=1}^n$ be the eigenvalues of $\bm{W}$ with associated eigenvectors $\{v_i\}_{i=1}^n$. Under Assumption 3, It is clear that 
%$-1<\lambda_n(\bm{W})\leq \dots\leq\lambda_2<\lambda_1(\bm{W})=1$. 

%The weight matrices are independent over time. But for any time $k$, each communication failure may be correlated. It is reasonable because interference among the wireless channels correlates the link failures over space while channels are independent over time \cite{66}. 
Eq.\eqref{con1} is equivalent to that the graph is connected in mean. 
It should be noted that a connected graph is important for all agents to achieve global optimal solution with only local communication. 
In \cite{68} and \cite{69}, the graph is assumed to be jointly connected, i.e., for any $k$, within constant $B$ intervals, the joint graph $\bigcup_{t=k}^{k+B}\mathcal{E}(t)$ is connected. In stochastic networks, due to the randomness of communication link failure, we cannot determine a constant $B$ such that the joint graph is connected with every $B$ intervals. 
In this paper, we only require the the graph to be connected in mean, which means that the communication network may be disconnected at each interval. This is different from the assumption of bounded intercommunication interval in \cite{68,69}.
%and only assume that the graph is connected in mean, i.e., $\mathbb{E}\{\mathcal{E}(k)\}$ is connected.
\iffalse
\begin{rem}
\emph{Assumption \ref{network} requires the network to be symmetric, doubly stochastic for any time $k$. Here, we provide a method to realize doubly stochastic network even with communication failure.}

\emph{For each time $k$, agent $i$ sends information to agent $j$, along with the weight $w_i^j(k)$, $j\in\mathcal{N}_i$.  If the communication link between $i$ and $j$ works, agent $i$ and $j$ will receive $w_j^i(k)$ and $w_i^j(k)$, respectively. The weight of the edge $(i, j)$ is set as the smaller one between $w_i^j(k)$ and $w_j^i(k)$, i.e.,}
\begin{align}
w_{ij}(k)=
\begin{cases}
\min\{w_i^j(k), w_j^i(k)\} & \text{if link}\  (i, j)\  \text{works}
\\0& \text{if link}\  (i, j)\ \text{fails}
\end{cases}.
\end{align}
\emph{Then, for each agent $i$, }
\begin{align}
w_{ii}(k)=1-\sum_{j\in\mathcal{N}_i}w_{ij}(k),
\end{align}
\emph{and thus $\sum_{j=1}^{n} w_{ij}(k)=1$ and $\sum_{j=1}^{n}w_{ji}(k)=1, \forall i, k$.}
\end{rem}
\fi
%they set the value of $w_{ii}$ to be $1-\sum_{j\in\mathcal{N}_i}w_{ij}$

%It should be noted that a connected graph is important for all agents to achieve global optimal solution with only local communication. In this paper, we only require the the graph to be connected in mean, i.e., the network $\mathbb{E}\{\bm{W}(k)\}$ is connected, which means that the communication network may be disconnected at each interval. 

Let $\bar{\lambda}_{2e}=\max_{k}\{\lambda_2(\mathbb{E}\{\bm{W}(k)\})\}$, $\underline{\lambda}_{ne}=\min_{k}\{\lambda_n(\mathbb{E}\{\bm{W}(k)\})\}$. Due to \eqref{con1}, we have that $-1<\underline{\lambda}_{ne}\leq\bar{\lambda}_{2e}<1$.
Moreover, let $\underline{\lambda}_n=\inf_{k}\{\lambda_n(\bm{W}(k))\}$.
Since $\bm{W}(k)$ is doubly stochastic, $w_{ii}(k)>0$ and $w_{ij}(k)\geq0$, we have $\underline{\lambda}_n>-1$.
%It is not difficult to realize because if the communication network is connected when there is no link failure, denoted by $\bm{W}^0(k)$, then, for each $w^0_{ij}(k)>0$, we have $\mathbb{E}(w_{ij}(k))=w^0_{ij}(k)\theta_{ij}>0$, implying that $\mathbb{E}\{\bm{W}(k)\}$ is still connected. 
%It should be noticed that different from other works which require the network to be connected, 
%the communication network in our work may be disconnected at each time. 
%Additionally, the network can also be directed as long as it satisfies Assumption 3. 
\section{Algorithm development}
\subsection{Optimal conditions}
Before introducing our algorithm, we give the conditions of the optimal solution. 
%based on which, we then impose another assumption about the existence of the optimal point.
\begin{prop}\label{optimalcondition}
$\bm{x}^*=[{x}_1^{*},...,{x}_n^{*}]^T\in\mathbb{R}^{n\times u}$ is the optimal solution of problem \eqref{problem1} if $\bm{x}^*$ satisfies the following two conditions.
\begin{enumerate}
\item[(i)] 
$\bm{1}^T\bm{x}^*=\bm{1}^T\bm{d},$
\item[(ii)] 
there exists ${{\mu}}^*\in\mathbb{R}^{u}$ such that $\nabla\bm{f}(\bm{x}^*)=\bm{1}({{\mu}}^*)^T$.
\end{enumerate}
\end{prop}
\begin{proof}
Define the Lagrange function 
\begin{align}
L(\bm{x}, {\mu})=\sum\limits_{i=1}^nf_i({x}_i)+{\mu}\left(\sum_{i=1}^{n}{x}_i-\sum_{i=1}^{n}{d}_i\right),
\end{align}
where ${\mu}\in\mathbb{R}^u$ is the Lagrange multiplier. 
%Then, let $\frac{\partial L(\bm{x}, \bm{\mu})}{\partial \bm{x}_i}=0, \forall i$ and $\frac{\partial L(\bm{x}, \bm{\mu})}{\partial \bm{\mu}}=0, \forall i$, 
By the KKT conditions \cite{60}, we obtain that 
\begin{align}
\nabla f_i({x}_i^*)-{\mu}^*=0, \forall i, \ \ \sum_{i=1}^{n}{x}_i^*-\sum_{i=1}^{n}{d}_i=0.\label{opcon2}
\end{align}
Then, we obtain that
%\begin{align}
%\nabla f_i(\bm{x}^*_i)=\nabla f_j(\bm{x}^*_j), \forall i,j\in\{1,...,n\},\notag
%\end{align}
%which means that
\begin{align}
\nabla\bm{f}(\bm{x}^*)=\left[\nabla f_1({x}^*_1),...,\nabla f_n({x}^*_n)\right]^T=\bm{1}({\mu}^{*})^T.\label{opconproof1}
\end{align}
Eqs.\eqref{opcon2} and \eqref{opconproof1} verify the conditions (i) and (ii), respectively, completing the proof.
\end{proof}
Proposition \ref{optimalcondition} implies that the optimal point is achieved when all the states are feasible and simultaneously the gradients of all agents' functions are equal, i.e., achieving consensus on marginal costs while keeping the balance between supply and demand. %We give the assumption about the existence of the optimal solution.
\subsection{Distributed deviation-tracking method}
%Consider the distributed resource allocation problem, which is shown as 
%\begin{align}
%\min_{\bm{x}\in\mathbb{R}^{n\times u}} f(\bm{x})=\sum_{i=1}^nf_i(x_i)\ \ s.t. \sum_{i=1}^{n}x_i=\sum_{i=1}^{n}d_i.
%\label{problem1exa}
%\end{align}
To solve the problem \eqref{problem1}, we develop a distributed resource allocation algorithm by combining weighted gradient method and deviation-tracking scheme. 

First, we adopt weighted gradient method to ensure that the convergence point satisfies Condition (ii) of Proposition 1. The update rule is shown as 
\begin{align}
x_i(k+1)=x_i(k)-\beta\sum_{j=1}^{n}w_{ij}(k)(\nabla f_i(x_i(k))-\nabla f_j(x_j(k))).\label{dt11}
\end{align}
This rule is also adopted in other distributed resource allocation problems over fixed networks \cite{46,47}.
For the Condition (i) of Proposition 1, we consider the following iteration.
\begin{align}
x_i(k+1)=&x_i(k)-\frac{\alpha}{n}\sum_{j=1}^{n}(x_j(k)-d_j).\label{dt12}
\end{align}
It is clear that if $x_i(k)$ converges, the deviation $\sum_{j=1}^{n}(x_j(k)-d_j)$ converges to 0, which is equivalent to Condition (i) of Proposition 1.
However, it is noted that the variable $\frac{1}{n}\sum_{j=1}^{n}(x_j(k)-d_j)$ requires all agents' states at each iteration, which is a global variable that cannot be obtained by each agent. We propose a deviation-tracking method to decentralize the global deviation $\frac{1}{n}\sum_{j=1}^{n}(x_j(k)-d_j)$ by introducing an auxiliary variable $y_i(k)$ to track it.
We attempt to use consensus protocol to ensure $y_i(k)$ tracks $\sum_{i=1}^{n}{x}(k)-{d}_i$. 
%Starting from $y_i(0)=x_i(0)-d_i$, 
%where $y_i(k)$ is updated according to ${y}_i(k+1)=\sum_{j=1}^{n}w_{ij}(k){y}_j(k)$.
But based on consensus protocol, the sum of $\sum_{i=1}^{n}{y}_j(k)$ is constant for all $k$.
% and equals to $\sum_{i=1}^{n}{y}_j(0)$. 
Now that $\frac{1}{n}\sum_{j=1}^{n}{x}_j(k)$ is time-varying, so we add a compensation term to track the value $\frac{1}{n}\sum_{j=1}^{n}{x}_j(k)$.
Starting from ${y}_i(0)={x}_i(0)-{d}_i$, the update rule of agent $i$ is expressed as follows.
\begin{equation}
\begin{split}
{x}_i(k+1)=&{x}_i(k)-\alpha{y}_i(k)
\\&-\beta\sum_{j=1}^{n}w_{ij}(k)(\nabla f_i(x_i(k))-\nabla f_j(x_j(k)))
\\ {y}_i(k+1)=&\sum_{j=1}^{n}w_{ij}(k){y}_j(k)+{x}_i(k+1)-{x}_i(k),
\end{split}
\label{statetrack}
\end{equation}
where $\alpha$ and $\beta$ are constant stepsizes. ${y}_i(k+1)$ is the auxiliary variable tracking the global deviation $\frac{1}{n}\sum_{j=1}^{n}(x_j(k)-d_j)$.

%\subsection{Distributed Communication Network Model}
%In this problem, we consider that all the agents do not know any information about the global network and other agents, which means that in problem \eqref{problem1}, agent $i$ only has local knowledge including its individual function $f_i$, local amount of resource $d_i$ and the local communication network with its neighbors. 
The algorithm is shown in Algorithm 1 for details.
\begin{algorithm}[h]
	\caption{: \textbf{Distributed Deviation-tracking Resource Allocation Algorithm}}
	\hspace*{0.02in} {\bf Input:}
	$\bm{W}, f_i, \alpha, \beta, {d}_{i}$\\
	\hspace*{0.02in} {\bf Output:}
	${x}_i,{y}_i$
	\begin{algorithmic}[1]
		\State \textbf{Initialization:} \small{Pick any ${x}_i(0)\in\mathbb{R}^u,\ {y}_i(0)={x}_i(0)-{d}_i$}
		\For{$k=0:\infty$}
		%\State Add noise to $x_i(k)$ and $y_i(k)$, respectively, and then obtain $z_{\zeta i}(k)$ and $z_{\eta i}(k)$;
		%\State \ \ \ \ $z_{\zeta i}(k)=x_i(k)+\zeta_i(k),\ z_{\eta i}(k)=y_i(k)+\eta_i(k)$;
		\State Broadcast $\nabla{f}_i({x}_i(k))$ and ${y}_i(k)$ to neighbors;
		\State Receive $\nabla{f}_j({x}_j(k))$ and ${y}_j(k), j\in \mathcal{N}_i$ from neighbors;
		%\begin{center} $z_{\zeta j}(k), z_{\eta j}(k), j\in \mathcal{N}_i$; \end{center}
		\State Update $x_i(k+1)$ through 
\State${x}_i(k+1)={x}_i(k)-\alpha{y}_i(k)$
\begin{flushright}$-\beta\Big[\nabla f_i({x}_i(k))-\sum_{j=1}^{n}w_{ij}(k)\nabla f_j({x}_j(k))\Big];$\end{flushright}
		%\begin{center} $x_i(k+1)=\sum_{j\in \mathcal{N}_i}a_{ij}z_{\zeta j}(k)-\alpha y_i(k)$; \end{center}
		\State Update $y_i(k+1)$ through
		\begin{center}${y}_i(k+1)=\sum_{j=1}^{n}w_{ij}(k){y}_j(k)+{x}_i(k+1)-{x}_i(k);$\end{center}
	
		%\\\ \ \qquad\qquad $y_i(k+1)=\sum_{j\in \mathcal{N}_i}a_{ij}z_{\eta j}(k)$ \\\qquad\qquad\qquad\qquad\quad\ \ \  $+\nabla f_i(x_i(k+1))-\nabla f_i(x_i(k))$.
		\EndFor
		\State \textbf{end for}
	\end{algorithmic}
\end{algorithm}

Let  
\begin{align}
\nabla\bm{f}(\bm{x}(k))=\left[\nabla f_{1}(\bm{x}_1(k)),...,\nabla f_{n}(\bm{x}_n(k))\right]^T\in\mathbb{R}^{n\times u}, \forall k.
\notag
\end{align}
Then, the algorithm can be rewritten in a matrix form:
\begin{align}
&\bm{y}(0)=\bm{x}(0)-\bm{d}\label{al0},
\\&\bm{x}(k+1)=\bm{x}(k)-\alpha\bm{y}(k)-\beta(\bm{I}-\bm{W}(k))\nabla\bm{f}(\bm{x}(k))\label{al1},
\\&\bm{y}(k+1)=\bm{W}(k)\bm{y}(k)+\bm{x}(k+1)-\bm{x}(k)\label{al2}.
\end{align}

%Traditional weighted gradients method in \cite{46,47} and the dual splitting approach in \cite{36} require the states to be feasible all through iterations, i,e, $\bm{1}^T\bm{x}(k)=\bm{1}^T\bm{d}, \forall k$. 
To intuitionally describe the algorithm of deviation-tracking method, we provide the block diagram of our algorithm. 
%and traditional weighted algorithm.
%As shown in Fig.1, traditional weighted gradients method 
%adjusts the states $x_i(k)$ by using the error caused by unequal gradients of all agents, i.e., $(\bm{I}-\bm{W}(k))\nabla\bm{f}(\bm{x}(k))$. This makes the convergence point satisfy the condition (ii) of Proposition \ref{optimalcondition}. In order to meet the condition (i) of Proposition \ref{optimalcondition}, it requires the initial states to be feasible and the states should be feasible all through the iteration i.e., $\bm{1}^T\bm{x}(k)=\bm{1}^T\bm{d}, \forall k$. However, if there exists any disturbance injected to the states, all agents' states cannot return to feasible region. The dual splitting approach proposed in \cite{36} also has this drawback.
In Algorithm 1, we adopt two feedback loops to ensure the optimal convergence. As shown in Fig. 1, we first use the $(\bm{I}-\bm{W}(k))\nabla\bm{f}(\bm{x}(k))$ to guarantee that the states converge with equal gradients because $\nabla f_i(x_i(k))=\nabla f_j(x_j(k)), \forall i,j$ iff
$(\bm{I}-\bm{W}(k))\nabla\bm{f}(\bm{x}(k))=0$.
Then, we introduce $\bm{y}(k)$ to track the deviation $\bm{1}^T\bm{x}(k)-\bm{1}^T\bm{d}$ and eliminate it by adjusting the state with $\bm{y}(k)$ to ensure the convergence point is feasible.
%adjust the state based on both $\bm{y}(k)$ and $(\bm{I}-\bm{W}(k))\nabla\bm{f}(\bm{x}(k))$. 
By adopting two variables feedback to lead the state to the point satisfying the Conditions (i) and (ii) in of Proposition \ref{optimalcondition}, the algorithm can converge to optimal solution even with infeasible initialization and disturbance on states, which will be shown in Section V.
\iffalse
In Algorithm 1, both $\bm{x}(k)$ and $\bm{y}(k)$ need to be sent to neighbors. They share the same communication network 
%and we do not need to generate another network specially for $\bm{y}(k)$. 
and thus suffer from the same kind of 
%But the transmissions of both $\bm{x}(k)$ and $\bm{y}(k)$ suffer 
random communication failures. 
%Here we assume that $\bm{W}(k)$ does not change during the time $k$, which suggests that if $x_{i}(k)$ is sent to agent $j$ successfully,  the transmission of $y_{i}(k)$ to $j$ must be successful. 
In fact, the algorithm can be extend to the situation that the communication networks for transmitting $\bm{x}(k)$ and $\bm{y}(k)$ are different as long as they satisfy Assumption \ref{network}.
\fi
%They share the same communication network and do not need two different 

%use the gradients consensus error

%the proposed Algorithm 1 does not need the feasible initialization. By introducing the deviation tracking step \eqref{statetrack}, we eliminate the initialization deviation $\bm{\epsilon}=\bm{1}^T\bm{x}(0)-\bm{1}^T\bm{d}$ in a distributed manner and the error converges to 0 exponentially. Similar to the initialization deviation, the deviation caused by disturbance can also be eliminated, which will be shown in Section VI.
%\begin{figure}
%		\centering
%		\includegraphics[scale=0.5]{flowchartnew2.pdf}
%\caption{Block diagram of traditional weighted gradients method} 
%\end{figure}
\begin{figure}
		\centering
		\includegraphics[scale=0.4]{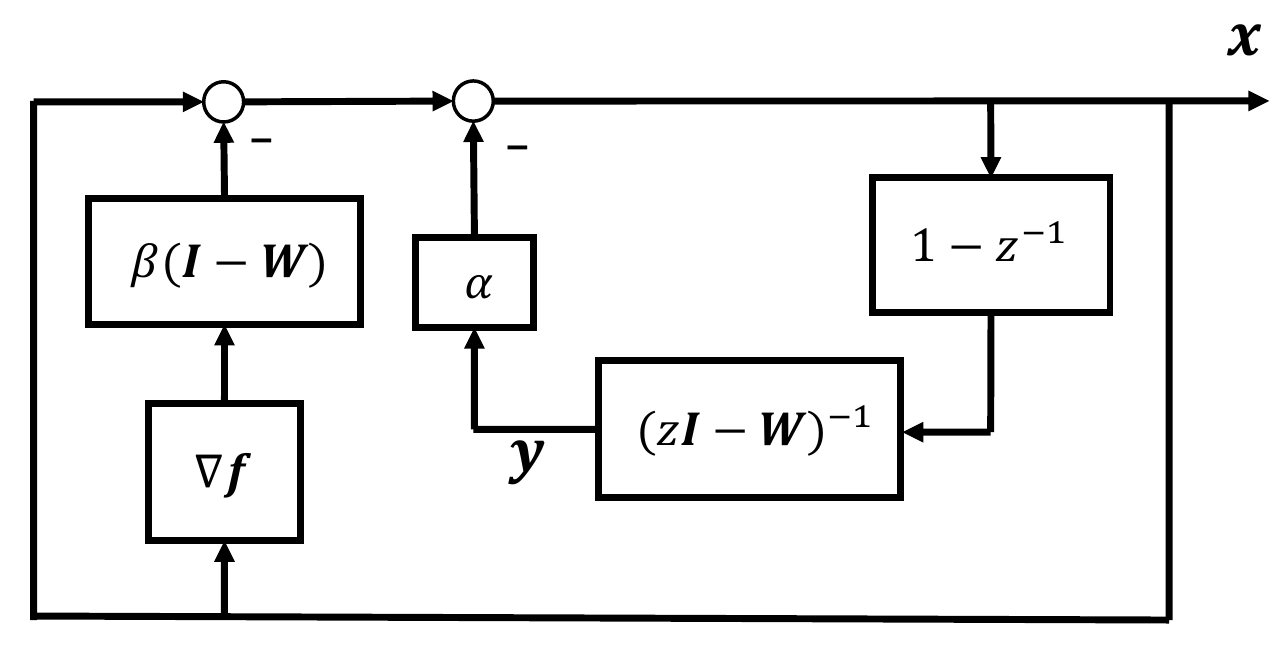}
\caption{Block diagram of Algorithm 1.} 
\end{figure}

%The communication network is subject to random failure. 
%Since the network is undirected, if the edge $(i,j)$ fails at time $k$, both agents $i$ and $j$ cannot receive the message from each other. Additionally, if agent $j$ cannot receive any neighbors' messages, i.e. $\forall i\in\mathcal{N}_j$, $(i,j)$ fail, agent $j$ will set $x_{j}(k+1)=x_{j}(k)$ and $y_j(k+1)=y_j(k)$.

\iffalse
agent $j$ cannot receive the message from agent $i$, but agent $i$ can still receive agent $j$'s messages if $(j,i)$ works. Therefore, the communication network can be directed. If agent $j$ cannot receive any neighbors' messages, i.e. $\forall i\in\mathcal{N}_j$, $(i,j)$ fail, agent $j$ will set $x_{j}(k+1)=x_{j}(k)$ and $y_j(k+1)=y_j(k)$.
\fi

It is worth noting that we adopt two different constant stepsizes, $\alpha$ and $\beta$ in \eqref{al0}-\eqref{al2}, instead of decaying stepsize that used in \cite{49} and \cite{50}. In \cite{49}, the stepsize $\alpha(k)$ is required to satisfy $\sum_{k=1}^{\infty}\alpha(k)=\infty$ and $\sum_{k=1}^{\infty}(\alpha(k))^2<\infty$, which results in convergence rate of $\mathcal{O}(\ln(k)/\sqrt{k})$. Similar assumptions about stepsize are also imposed in \cite{50}. Although the algorithm proposed in \cite{50} converges linearly, it is only suitable for fixed networks and there is always a steady state error between the convergence point and the optimal solution. In next section, we will show that the proposed algorithm converges to the optimal solution exactly at a linear convergence rate with constant stepsize even over stochastic networks.

\section{Convergence analysis over stochastic network}
In this section, we will show that the proposed Algorithm 1 converges linearly to the optimal solution in mean square over stochastic communication networks where each communication link fails randomly. 
Without loss of generality, we consider the situation when $u=1$. Extension to the case of $u>1$ is straightforward.
%In this section, we consider the scenario that each network communication link fails randomly and the whole network may be disconnected at each time. Therefore, the condition that $\rho(\bm{W}(k)-\frac{\bm{11}^T}{n})<1, \forall k$ does not hold any more. Here, we prove that the algorithm still converges to the optimal solution in mean square as long as Assumption 3 holds.
\begin{lem}
Let $\bm{\Gamma}=diag\{\gamma_1,...,\gamma_n\}$ and $\bm{T}=\bm{L}\bm{\Gamma}(\bm{I}-\bm{A})$, where $\eta_i\leq\gamma_i\leq\varphi_i, \forall i$. Suppose Assumptions 1-2 hold and $\bm{A}$ is doubly stochastic. For any vector $\bm{v}$ such that $\bm{1}^T\bm{v}=0$ and $\bm{v}\neq\bm{0}$, we have
\begin{align}
\frac{\underline{\eta}(1-\lambda_2(\bm{A}))\left[\overline{\varphi}+(n-1)\underline{\eta}\right]}{\sqrt{n\left[\overline{\varphi}^2+(n-1)\underline{\eta}^2\right]}}&\leq\frac{\Vert\bm{T}\bm{v}\Vert_2}{\Vert\bm{v}\Vert_2}\leq\overline{\varphi}(1-{\lambda}_n(\bm{A})),\label{lem1p1}
\end{align}
where $\underline{\eta}=\min_i\{\eta_i\}$, $\overline{\varphi}=\max_i\{\varphi_i\}$.
%, $\lambda_2=\lambda_2\{A\}$ and $\lambda_n=\lambda_n\{\bm{A}\}$.
\end{lem}

\begin{lem}\label{lem1}
Let $K_1=\frac{\underline{\eta}(1-\bar{\lambda}_{2e})\left[\overline{\varphi}+(n-1)\underline{\eta}\right]}{\sqrt{n\left[\overline{\varphi}^2+(n-1)\underline{\eta}^2\right]}}$, $K_2=\overline{\varphi}(1-\underline{\lambda}_{ne})$ and $\bm{T}_s=\bm{L}\bm{\Gamma}(\bm{I}-\bm{W})$,
where $\bm{W}$ is a stochastic matrix satisfying \eqref{doublestochastic} and \eqref{con1}. For any vector $\bm{v}$ such that $\bm{1}^T\bm{v}=0$ and $\bm{v}\neq\bm{0}$, we have
\begin{align}
\Vert(\bm{I}-\beta\bm{T}_s)\bm{v}\Vert_E^2\leq (1+\beta^2K_2'^2-2\beta K_1')\Vert\bm{v}\Vert_E^2.\notag
\end{align}
\end{lem}
%Proofs of Lemmas 1 and 2 are shown in Appendix A and B.
%This together with \eqref{lem14} shows that $\bm{v}^T\bm{T}\bm{v}>0$, completing the proof.  
\begin{lem}\label{lem1}
Suppose Assumption \ref{network} holds, we have
\begin{align}
\rho\left(\mathbb{E}\{(\bm{W}(k))^2\}-\frac{\bm{11}^T}{n}\right)<1.\label{lemm3}
\end{align}
\end{lem}
Proofs of Lemmas 1, 2 and 3 are shown in Appendixes A, B and C, respectively.
\begin{thm}\label{nconsto}
Let $\overline{\lambda}_{2s}=\max_{k}\{\lambda_2(\mathbb{E}\{(\bm{W}(k))^2)\}$, $s_1=\alpha^2+2\alpha+\overline{\lambda}_{2s}$ and $s_2=\sqrt{1+\beta^2K_2^2-2\beta K_1}$. 
%where $K_1=\frac{\underline{\eta}(1-\lambda_2)\left[\overline{\varphi}+(n-1)\underline{\eta}\right]}{\sqrt{n\left[\overline{\varphi}^2+(n-1)\underline{\eta}^2\right]}}$ and $K_2=\overline{\varphi}(1-\lambda_n)$.
%\mathcal{S}_{\alpha}=\left\{(\alpha,\beta)|\alpha<1-\lambda_n, \beta<\frac{1}{\overline{\varphi}(1-\lambda_n)}, \\\alpha\beta\overline{\varphi}(1-\lambda_n)<(1-s_1)(1-s_2)\right\}.
Consider the sequence $\{\bm{x}(k)\}_{k=0}^{\infty}$ and $\{\bm{y}(k)\}_{k=0}^{\infty}$ generated by Algorithm 1. Suppose Assumptions \ref{lip}-\ref{network} hold. For any $\bm{x}(0)\in\mathbb{R}^{n}$ and $(\alpha,\beta)$ satisfying
\begin{equation}
\begin{split}
&\alpha<\sqrt{2-\overline{\lambda}_{2s}}-1, \ \beta<\frac{2K_1}{K_2^2},
\\&\alpha\beta\overline{\varphi}(1-\underline{\lambda}_n)<(1-s_1)(1-s_2),
\end{split}
\label{nafarange}
\end{equation}
there exist $m_{1s}, m_{2s}>0$ and $q_s\in(0,1)$ such that for any $k\geq0$,
\begin{align}
%&\Vert\bm{x}(k+1)-\bm{x}(k)\Vert\leq m_1q^k\label{res1};
&\Vert\bm{y}(k)\Vert_E\leq m_{1s}q_s^k\label{nres2},
\\&\Vert\nabla\check{\bm{f}}(\bm{x}(k))\Vert_E\leq m_{2s}q_s^k\label{nres3},
\end{align}
where $\nabla\check{\bm{f}}(\bm{x}(k))= \bm{L}\nabla\bm{f}(\bm{x}(k))$.
%where $\nabla\check{\bm{f}}(\bm{x}(k))= \bm{L}\nabla\bm{f}(\bm{x}(k))$.
\end{thm}
\begin{proof}

First, we multiply both sides of \eqref{al1} and \eqref{al2} by $\frac{\bm{11}^T}{n}$, which yields
\begin{align}
&\bar{\bm{x}}(k+1)=\bar{\bm{x}}(k)-\alpha\bar{\bm{y}}(k),\label{c11}
\\&\bar{\bm{y}}(k+1)=\bar{\bm{y}}(k)+\bar{\bm{x}}(k+1)-\bar{\bm{x}}(k).\label{c22}
\end{align}
Due to $\bm{y}(0)=\bm{x}(0)-\bm{d}$, we have $\bar{\bm{y}}(k)=\bar{\bm{x}}(k)-\bar{\bm{d}},\forall k$.
Based on \eqref{c11}, we have 
$
\bar{\bm{x}}(k+1)-\bar{\bm{d}}=(1-\alpha)(\bar{\bm{x}}(k)-\bar{\bm{d}})
$.

\iffalse
Hence, 
\begin{align}
\Vert\bar{\bm{y}}(k)\Vert_2&=\Vert\bar{\bm{x}}(k+1)-\bar{\bm{d}}\Vert_2\notag
\\&\leq (1-\alpha)^k\Vert\bar{\bm{x}}(0)-\bar{\bm{d}}\Vert_2+\sum_{t=0}^{k-1}(1-\alpha)^{k-1-t}w_{\zeta}q_{\zeta}^{t}\notag
\\&=(1-\alpha)^k\Vert\bar{\bm{x}}(0)-\bar{\bm{d}}\Vert_2+w_{\zeta}\frac{(1-\alpha)^{k}-q_{\zeta}^{k}}{1-\alpha-q_{\zeta}}
\notag
\end{align}
\fi
We can see that for any $m_{0s}>\Vert\bar{\bm{x}}(0)-\bar{\bm{d}}\Vert_E$ and $q_{0s}\in[\vert1-\alpha\vert,1)$ such that 
\begin{align}
\Vert\bar{\bm{y}}(k)\Vert_E\leq m_{0s}q_{0s}^k,\label{barcon}
\end{align}
which means that $\bar{\bm{x}}(k)$ converges in mean square linearly to $\bar{\bm{d}}$ and $\bar{\bm{y}}(k)=\bar{\bm{x}}(k)-\bar{\bm{d}}$ converges to 0 as long as $0<\alpha<2$. Due to $\bm{y}(k)=\bar{\bm{y}}(k)+\check{\bm{y}}(k)$, we only need to prove that there exist $m_{1s}', m_{2s}'>0$ and $q'\in(q_0,1)$ such that
\begin{align}
%&\Vert\check{\bm{x}}(k+1)-\check{\bm{x}}(k)\Vert\leq m_1'q'^k\label{res11};
&\Vert\check{\bm{y}}(k)\Vert_E\leq m_{1s}'q_s'^k\label{res21},
\\&\Vert\nabla\check{\bm{f}}(\bm{x}(k))\Vert_E\leq m_{2s}'q_s'^k\label{res31},
\end{align}

\iffalse
First, we multiply both sides of \eqref{ald1} and \eqref{ald2} by $\frac{\bm{11}^T}{n}$ yields
\begin{align}
&\bar{\bm{x}}(k+1)=\bar{\bm{x}}(k)-\alpha\bar{\bm{y}}(k),\label{c11}
\\&\bar{\bm{y}}(k+1)=\bar{\bm{y}}(k)+\bar{\bm{x}}(k+1)-\bar{\bm{x}}(k).\label{c22}
\end{align}
Due to $\bm{y}(0)=\bm{x}(0)-\bm{d}$, we have 
\begin{align}
\bm{y}(k)=\bm{x}(k)-\bm{d},\forall k.\label{th1e1}
\end{align}
Based on \eqref{c11} and \eqref{th1e1}, we have 
\begin{align}
\bar{\bm{x}}(k+1)-\bar{\bm{d}}=(1-\alpha)(\bar{\bm{x}}(k)-\bar{\bm{d}})\notag
\end{align}
and there exists $m_0>0$ such that 
\begin{align}
\Vert\bar{\bm{y}}(k)\Vert_E\leq m_0(1-\alpha)^k,\label{barcon}
\end{align}
which means that $\bar{\bm{x}}(k)$ converges linearly to $\bar{\bm{d}}$ and $\bar{\bm{y}}(k)=\bar{\bm{x}}(k)-\bar{\bm{d}}$ converges linearly to 0 as long as $0<\alpha<2$. Due to $\bm{y}(k)=\bar{\bm{y}}(k)+\check{\bm{y}}(k)$, we only need to prove that there exist $m_1', m_2'>0$ and $q'\in[1-\alpha,1)$ such that
\begin{align}
%&\Vert\check{\bm{x}}(k+1)-\check{\bm{x}}(k)\Vert\leq m_1'q'^k\label{res11};
&\Vert\check{\bm{y}}(k)\Vert_E\leq m_1'q'^k\label{res21},
\\&\Vert\nabla\check{\bm{f}}(\bm{x}(k))\Vert_E\leq m_2'q'^k\label{res31},
\end{align}
\fi

From Assumptions 1 and 2, we obtain that for any $k$, there exists a diagonal matrix $\bm{\Theta}(k)=diag\{\theta_1(k),...,\theta_n(k)\}$, $\eta_i\leq\theta_i(k)\leq\varphi_i, \forall i, k$, such that 
\begin{align}
\nabla\bm{f}(\bm{x}(k+1))-\nabla\bm{f}(\bm{x}(k))=\bm{\Theta}(k)(\bm{x}(k+1)-\bm{x}(k)).\label{th1e2}
\end{align}
Substituting \eqref{th1e2} into \eqref{al1}, we have 
\begin{multline}
\nabla\bm{f}(\bm{x}(k+1))-\nabla\bm{f}(\bm{x}(k))=-\alpha\bm{\Theta}(k)\bm{y}(k)
\\-\beta\bm{\Theta}(k)(\bm{I}-\bm{W}(k))\nabla\bm{f}(\bm{x}(k)).\label{th1e3}
\end{multline}
Multiplying both sides of \eqref{th1e3} by $\bm{L}$, we obtain 
\begin{multline}
\nabla\check{\bm{f}}(\bm{x}(k+1))-\nabla\check{\bm{f}}(\bm{x}(k))=-\alpha\bm{L}\bm{\Theta}(k)\bm{y}(k)
\\-\beta\bm{L}\bm{\Theta}(k)(\bm{I}-\bm{W}(k))\nabla\check{\bm{f}}(\bm{x}(k))\label{th1e4},
\end{multline}
where $(\bm{I}-\bm{W}(k))\nabla\check{\bm{f}}(\bm{x}(k))=(\bm{I}-\bm{W}(k))\bm{L}\nabla\bm{f}(\bm{x}(k))=(\bm{I}-\bm{W}(k))\nabla\bm{f}(\bm{x}(k))$.
From \eqref{th1e4}, we further obtain 
\begin{align}
\nabla\check{\bm{f}}(\bm{x}(k+1))=&\left(\bm{I}-\beta\bm{L}\bm{\Theta}(k)(\bm{I}-\bm{W}(k))\right)\nabla\check{\bm{f}}(\bm{x}(k))\notag
\\&-\alpha\bm{L}\bm{\Theta}(k)\check{\bm{y}}(k)-\alpha\bm{L}\bm{\Theta}(k)\bar{\bm{y}}(k).\label{th1e5}
%\\&+\bm{L}\bm{\Theta}(k)\bm{\zeta}(k)
\end{align}
Then, from \eqref{al2}, we have 
\begin{align}
\check{\bm{y}}(k+1)=\bm{W}(k)\check{\bm{y}}(k)+\check{\bm{x}}(k+1)-\check{\bm{x}}(k).\label{th1e6}
\end{align}
Multiplying both sides of \eqref{al1} by $\bm{L}$, we have
\begin{align}
\check{\bm{x}}(k+1)=\check{\bm{x}}(k)-\alpha\check{\bm{y}}(k)-\beta(\bm{I}-\bm{W}(k))\nabla\check{\bm{f}}(\bm{x}(k)).\label{th1e7}
\end{align}
Substituting \eqref{th1e7} into \eqref{th1e6}, we have 
\begin{align}
\check{\bm{y}}(k+1)=&(\bm{W}(k)-\alpha\bm{I})\check{\bm{y}}(k)-\beta(\bm{I}-\bm{W}(k))\nabla\check{\bm{f}}(\bm{x}(k)).
\label{th1ex7}
\end{align}

Next, we will prove \eqref{res21} and \eqref{res31} by mathematical induction based on \eqref{th1e5} and \eqref{th1ex7}.

It is not difficult to prove that when $k=0$, there exist $m_{1s}', m_{2s}'>0$ and $q_s'\in(q_{0s},1)$ such that \eqref{res21} and \eqref{res31} hold.

Then, we assume that for all $k\leq\kappa$, \eqref{res21} and \eqref{res31} hold.

When $k=\kappa+1$,
from \eqref{th1ex7}, because $\bm{W}(k)$ is independent of $\bm{x}(k)$ and $\bm{y}(k)$, 
we have 
\iffalse
\begin{align}
\Vert\check{\bm{x}}(k+1)-\check{\bm{x}}(k)\Vert_2&\leq\alpha\Vert\check{\bm{y}}(k)\Vert_2+\beta(1-\lambda_n)\Vert\nabla\check{\bm{f}}(\bm{x}(k))\Vert_2\notag
\\&\leq \alpha m_{1s}'q_s'^k+\beta(1-\lambda_n)m_{2s}'q_s'^k.\label{th1e8}
%\\\Vert\check{\bm{y}}(k+1)\Vert\leq\lambda_n\Vert\check{\bm{y}}(k)\Vert+\Vert\check{\bm{x}}(k+1)-\check{\bm{x}}(k)\Vert\notag
%\\\Vert\nabla\check{\bm{f}}(\bm{x}(k+1))\Vert\leq s\Vert\nabla\check{\bm{f}}(\bm{x}(k))\Vert+\alpha\eta_{\max}\Vert\check{\bm{y}}(k)\Vert
%\Vert\left(\bm{I}-\beta\bm{L}\Theta(k)(\bm{I}-\bm{W})\right)\Vert
\end{align}
\fi
%Based on \eqref{th1e6} and \eqref{th1e8}, we have 
\begin{align}
\Vert\check{\bm{y}}(k+1)\Vert_E&\leq s_1\Vert\check{\bm{y}}(k)\Vert_E+\beta(1-\underline{\lambda}_n)\Vert\nabla\check{\bm{f}}(\bm{x}(k))\Vert_E\notag
\\&\leq \left[s_1m_{1s}'+\beta(1-\underline{\lambda}_n)m_{2s}'\right]q_s'^k.\label{mth1ex8}
\end{align}
%where $s_1'=\max\{\lambda_2-\alpha, \alpha-\underline{\lambda}_n\}$.
From \eqref{th1e5} and \eqref{barcon}, we obtain based on Lemma \ref{lem1} that
\begin{align}
\Vert\nabla\check{\bm{f}}(\bm{x}(k+1))\Vert_E\leq &s_2\Vert\nabla\check{\bm{f}}(\bm{x}(k))\Vert_E\notag
\\&+\alpha\overline{\varphi}(\Vert\check{\bm{y}}(k)\Vert_E+\Vert\bar{\bm{y}}(k)\Vert_E)\notag
%\\\leq &s_2'm_{2s}'q_s'^k+\alpha\overline{\varphi}m_{1s}'q_s'^k+\alpha\overline{\varphi}\Vert\bar{\bm{y}}(k)\Vert_2\notag
\\\leq &\left[s_2m_{2s}'+\alpha\overline{\varphi}(m_{1s}'+m_0)\right]q_s'^k.\label{mth1ex9}
\end{align}
%From Lemma 1, we know that $s_2=\max\{1-\beta K_1, \beta K_2'-1\}$, where $K_1=\frac{\underline{\eta}(1-\lambda_2)\left[\overline{\varphi}+(n-1)\underline{\eta}\right]}{\sqrt{n\left[\overline{\varphi}^2+(n-1)\underline{\eta}^2\right]}}$ and $K_2=\overline{\varphi}(1-\underline{\lambda}_n)$.
%where $s_2=\max_k\{\rho\left(\bm{I}-\beta\bm{L}\bm{\Theta}(k)(\bm{I}-\bm{W})\right)\}$. From Lemma 1, we know that $s>0$ and $s$ does not depend on $k$.
%where $q_s'\geq1-\alpha$.
Due to $\alpha<\sqrt{2-\overline{\lambda}_{2s}}-1, \ \beta<\frac{2K_1}{K_2^2}$, we have $s_1<1$ and $s_2<1$.
Since $\alpha\beta\overline{\varphi}(1-\underline{\lambda}_n)<(1-s_1)(1-s_2)$,
there exist $m_{1s}'>0$ and $m_{2s}'>0$ such that 
\begin{align}
\alpha\beta\overline{\varphi}(1-\underline{\lambda}_n)(m_{1s}'+m_{0s})&<\beta(1-\underline{\lambda}_n)(1-s_2)m_{2s}'\notag
\\&< (1-s_1)(1-s_2)m_{1s}'.\label{mss1}
\end{align}
Then, we have there exists $q_s'\in(q_0, 1)$ such that
\begin{align}
\left[s_1m_{1s}'+\beta(1-\underline{\lambda}_n)m_{2s}'\right]q_s'^k\leq m_{1s}'q_s'^{k+1},\label{mth1ex10}
\\\left[s_2m_{2s}'+\alpha\overline{\varphi}(m_{1s}'+m_{0s})\right]q_s'^k\leq m_{2s}'q_s'^{k+1},\label{mth1ex11}
\end{align}
Substituting \eqref{mth1ex10} and \eqref{mth1ex11} into \eqref{mth1ex8} and \eqref{mth1ex9} yields
\begin{align}
&\Vert\check{\bm{y}}(k+1)\Vert_E\leq m_{1s}'q_s'^{k+1},\notag
\\&\Vert\nabla\check{\bm{f}}(\bm{x}(k+1))\Vert_E\leq m_{2s}'q_s'^{k+1}.\notag
\end{align}
This completes the induction.
\iffalse
Then, based on \eqref{nth1ex8} and \eqref{nth1ex9}, we only need to prove that there exist $m_{1s}', m_{2s}'>0$ and $q_s'\in(1-\alpha,1)$ such that 
\begin{align}
\left[s_1'm_{1s}'+\beta(1-\underline{\lambda}_n)m_{2s}'+m_{\zeta}\right]q_s'^k\leq m_{1s}'q_s'^{k+1},\label{nth1ex10}
\\\left[s_2'm_{2s}'+\alpha\overline{\varphi}(m_{1s}'+m_0)+\overline{\varphi}m_{\zeta}\right]q_s'^k\leq m_{2s}'q_s'^{k+1},\label{nth1ex11}
\end{align}
which is equivalent to showing that there exist $m_{1s}', m_{2s}'>0$ such that
\begin{align}
s_1'm_{1s}'+\beta(1-\underline{\lambda}_n)m_{2s}'+m_{\zeta}< m_{1s}',\label{th1e9}
\\s_2'm_{2s}'+\alpha\overline{\varphi}(m_{1s}'+m_0)+\overline{\varphi}m_{\zeta}< m_{2s}'.\label{th1e10}
\end{align}
Combining \eqref{th1e9} and \eqref{th1e10}, we have
\begin{align}
&\alpha\beta\overline{\varphi}(1-\underline{\lambda}_n)(m_{1s}'+m_0)+\beta(1-\underline{\lambda}_n)\overline{\varphi}m_{\zeta}
\\&<\beta(1-\underline{\lambda}_n)(1-s_2')m_{2s}'\notag
\\&< (1-s_1')(1-s_2')m_{1s}'-(1-s_2')m_{\zeta}.\notag
\end{align}
If we make $m_{1s}'$ and $m_{2s}'$ large enough, $m_0$ and $m_{\zeta}$ will be neglected. We only need to ensure that 
\begin{align}
\alpha\beta\overline{\varphi}(1-\underline{\lambda}_n)<(1-s_1')(1-s_2').\label{th1e11}
\end{align}
Additionally, from \eqref{th1e11}, we know that $s_1'$ and $s_2'$ should be less than $1$. Therefore, $\alpha$ and $\beta$ should also satisfy 
\begin{align}
\alpha<1-\underline{\lambda}_n; \beta<\frac{1}{\overline{\varphi}(1-\underline{\lambda}_n)},
\end{align}
completing the induction.\fi

Then, based on \eqref{barcon} and \eqref{res21}, we obtain that 
\begin{align}
\Vert\bm{y}(k)\Vert_E\leq\Vert\bar{\bm{y}}(k)\Vert_E+\Vert\check{\bm{y}}(k)\Vert_E\leq(m_{0s}+m_{1s}')q_s'^k,
\end{align}
which validates \eqref{nres2}. The proof is completed.
\end{proof}
%and we obtain the range of $\alpha$, i.e.,
%\begin{align}
%\alpha<\frac{(1-s)(1-\lambda_n)}{\beta\overline{\varphi}(1-\lambda_n)+(1-s)},
%\end{align}
\iffalse
where
\begin{align}
1-s=\max\left\{\frac{\beta\underline{\eta}(1-\lambda_2)\left[\overline{\varphi}+(n-1)\underline{\eta}\right]}{\sqrt{n\left[\overline{\varphi}^2+(n-1)\underline{\eta}^2\right]}}, 2-\beta\overline{\varphi}(1-\lambda_n)\right\}
\end{align}
\fi
\begin{rem}
In Theorem 1, the stepsize $\alpha$ and $\beta$ are limited by \eqref{nafarange}. The upperbound of $\alpha$ is related to $\beta$ and vice versa. Since $\lim_{\alpha\rightarrow0}\frac{\alpha}{1-s_1}=0$ and $\lim_{\beta\rightarrow0}\frac{1-s_2}{\beta}=2K_1$, as long as $\alpha$ and $\beta$ small enough, the equation \eqref{nafarange} hold and there exist $\alpha$ and $\beta$ such that the proposed algorithm converges in mean square.

It should also be noted that $m_{1s}$, $m_{2s}$ and $q_{s}$ are independent of $k$. If $m_{1s}$ is a sufficiently large number, $m_2$ can be selected as $\frac{\alpha\beta\overline{\varphi}(1-\underline{\lambda}_n)+(1-s_1)(1-s_2)}{2\beta(1-\underline{\lambda}_n)(1-s_2)}m_{1s}$ and $q_s$ is any number in the interval $(q_{sl}, 1)$, where $q_{sl}=\max\{s_1+\beta(1-\underline{\lambda}_n)\frac{m_{2s}}{m_{1s}}, s_2+\alpha\overline{\varphi}\frac{m_{1s}}{m_{2s}}, q_0\}$. $q_{sl}<1$ can be proved by \eqref{nafarange}. 
Therefore, $m_{1s}$, $m_{2s}$ and $q_{s}$ are all independent of $k$ and similar results is also applicable to the following Theorems 2 and 3.

Additionally, most existing works require the network to be connected or jointly connected, i.e., there exist an integer $B\geq1$ such that $\lambda_2\{\bm{W}(k)\bm{W}(k+1)...\bm{W}(k+B)\}<1$. But in stochastic networks, $B$ may not exist. In this paper, we only need the network is connected in mean, i.e., $\lambda_2\{\mathbb{E}\{\bm{W}(k)\}\}<1$.
\end{rem}

\begin{cor}
Suppose Assumptions 1-3 hold.
For any $\bm{x}(0)\in\mathbb{R}^{n}$ and $(\alpha,\beta)$ satisfying \eqref{nafarange},
the sequence $\{\bm{x}(k)\}_{k=0}^{\infty}$ generated by Algorithm 1 converges linearly to $\bm{x}^*$ in mean square. The convergence rate is 
\begin{align}
\max\Big\{&s_1+\frac{\alpha\beta\overline{\varphi}(1-\underline{\lambda}_n)}{1-s_2}, s_2+\frac{\alpha\beta\overline{\varphi}(1-\underline{\lambda}_n)}{1-s_1}, \vert1-\alpha\vert\Big\}.\notag
%\label{mlowerbounds}
\end{align} 
\end{cor}
\begin{proof}
Theorem \ref{fixedbounded} implies that $\Vert\bm{y}(k)\Vert_E$ and $\Vert\nabla\check{\bm{f}}(\bm{x}(k))\Vert_E$ converge to zero linearly. From \eqref{al0} and \eqref{al2}, we have $\bm{1}^T\bm{y}(k)=\bm{1}^T\bm{x}(k)-\bm{1}^T\bm{d}$. So the condition (i) in Proposition 1 is ensured if $\Vert\bm{y}(k)\Vert_E$ converges to 0 and so is the condition (ii) in Proposition 1 if $\Vert\nabla\check{\bm{f}}(\bm{x}(k))\Vert_E$ converges to 0. Therefore, $\bm{x}(k)$ converges to the optimal solution $\bm{x}^*$. Moreover, due to the linear convergence of $\Vert\bm{y}(k)\Vert_E$ and $\Vert\nabla\check{\bm{f}}(\bm{x}(k))\Vert_E$, we obtain from \eqref{al1} that $\bm{x}(k+1)-\bm{x}(k)$ converges linearly to 0, which, together with the above analysis, implies $\bm{x}(k)$ converges to $\bm{x}^*$ at a linear rate.

%Since Theorem \ref{consto} implies that $\Vert\bm{y}(k)\Vert_E$ and $\Vert\nabla\check{\bm{f}}(\bm{x}(k))\Vert_E$ converge to zero linearly, which validates the condition (i) and (ii) of Proposition \ref{optimalcondition}. Therefore, $\bm{x}(k)$ converges to $\bm{x}^*$ linearly in mean square. 

From \eqref{mth1ex10} and \eqref{mth1ex11}, we obtain that 
\begin{align}
q_s'\geq s_1+\beta(1-\underline{\lambda}_n)\frac{m_{2s}'}{m_{1s}'}, \ q_s'\geq s_2+\alpha\overline{\varphi}\frac{m_{1s}'}{m_{2s}'}.\label{mth2e1}
\end{align}
From \eqref{mss1}, we have 
\begin{align}
\frac{\alpha\overline{\varphi}}{1-s_2}<\frac{m_{2_s}'}{m_{1s}'}<\frac{1-s_1}{\beta(1-\underline{\lambda}_n)}.\label{mth2e3}
\end{align}
Since that $m_{1s}'$ and $m_{2s}'$ can be any positive value as long as they satisfy \eqref{mth2e3} and $q_s>\vert1-\alpha\vert$, we can get the result.
\end{proof}
\iffalse
Theorem 7 implies that $\Vert\bm{y}(k)\Vert_E$ and $\Vert\nabla\check{\bm{f}}(\bm{x}(k))\Vert_E$ converge to zero linearly. From \eqref{al0} and \eqref{al2}, we have $\bm{1}^T\bm{y}(k)=\bm{1}^T\bm{x}(k)-\bm{1}^T\bm{d}$. So the condition (i) in Proposition 4 is ensured if $\Vert\bm{y}(k)\Vert_E$ converges to 0 and so is the condition (ii) in Proposition 4 if $\Vert\nabla\check{\bm{f}}(\bm{x}(k))\Vert_E$ converges to 0. Therefore, $\bm{x}(k)$ converges to the optimal solution $\bm{x}^*$. Moreover, due to the linear convergence of $\Vert\bm{y}(k)\Vert_E$ and $\Vert\nabla\check{\bm{f}}(\bm{x}(k))\Vert_E$, we obtain from \eqref{al1} that $\bm{x}(k+1)-\bm{x}(k)$ converges linearly to 0, which, together with the above analysis, implies $\bm{x}(k)$ converges to $\bm{x}^*$ at a linear rate. 

\fi
%\begin{rem}
Compared with fixed networks, the stochastic network is time-varying and, most importantly, possible to be disconnected at any time, which is different from most of existing works that require the network to be always connected or jointly connected. 
%, which lead to $\rho(\bm{W}(k)-\frac{\bm{11}^T}{n})<1$. But this property is not possessed by stochastic network. 
Moreover, we obtain from Assumption \ref{network} that as long as $\theta_{ij}(k)>0$ for each communication link $(i,j)$, we will have the relation \eqref{con1}. This means that the proposed algorithm converges to the optimal solution even if each communication link works with a small positive possibility. Although small $\theta_{ij}(k)>0$ does not break the convergence of the algorithm, it affects the convergence rate, which will be extensively discussed in Section VIII.
%\end{rem}

\section{Resilience properties}
\subsection{Convergence in the presence of disturbances}
In this section, we shows that Algorithm 1 converges to the optimal solution in mean square under disturbance on states. Here we do not consider the disturbance on $\bm{y}(k)$ because $\bm{y}(k)$ is only a message that tracks the deviation while $\bm{x}(k)$ always denotes the running status of equipments, such as power output of generators, which is more likely to be disturbed by misoperation, attack, or environmental noise.

We add disturbance to states of \eqref{al0}-\eqref{al2} and obtain that
\begin{align}
\bm{x}(k+1)=&\bm{x}(k)+\bm{\zeta}(k)-\alpha\bm{y}(k)-\beta(\bm{I}-\bm{W}(k))\nabla\bm{f}(\bm{x}(k))\label{ald1},
\\\bm{y}(k+1)=&\bm{W}(k)\bm{y}(k)+\bm{x}(k+1)-\bm{x}(k)\label{ald2}. 
\end{align}
%where the fixed network matrix $\bm{W}$ is doubly stochastic and connected. Let $\lambda_2=\lambda_2(\bm{W})$ and 
%$\lambda_n=\lambda_n(\bm{W})$.
%Compared with traditional weighted gradient method, the proposed algorithm adds an auxiliary variable $\bm{y}(k)$ to track the deviation $\frac{1}{n}\sum_{j=1}^{n}(x_j(k)-d_j)$. Even if $\bm{x}(k)$ is disturbed, $\bm{y}(k)$ can still track the deviation and feed back to $\bm{x}(k)$ by \eqref{al1}. 
Random variable
$\bm{\zeta}(k)=[\zeta_1(k),...,\zeta_n(k)]$ is the disturbance injected to the state. We give the following assumption about the disturbance.
\begin{assum}\label{noi}
There exist $m_{\zeta}>0$ and $q_{\zeta}\in(0,1)$ such that for any $k$, $\Vert\bm{\zeta}(k)\Vert_E\leq m_{\zeta}q_{\zeta}^k$. 
\end{assum}

The resilience property means the algorithm states is able to maintain its convergence to the optimal solution after the injected disturbance decays. 
Assumption 4 restricts the second order moment of the disturbance to be covered by an upper bound that vanishes exponentially. 
%Actually, if the disturbance $\bm{\zeta}(k)$ does not converge to 0 as $k\rightarrow\infty$, it is impossible to ensure the convergence of the proposed algorithm. 
Assumption \ref{noi} is not only suitable for those continuous disturbance converging exponentially to 0, but it also applies to all kinds of finite-duration disturbances, i.e. the disturbance $\bm{\zeta}(k)=0$ after a finite-time steps $k\geq K_{\zeta}$.
Based on this assumption, we prove the linear convergence of the expectation of the states norm $\Vert\bm{x}(k)\Vert_E$.

\begin{thm}\label{consto}
%Let $s_1'=\max\{1-\alpha, \alpha-\underline{\lambda}_n\}$ and $s_2'=\sqrt{1+\beta^2K_2'^2-2\beta K_1'}$. 
%where $K_1=\frac{\underline{\eta}(1-\lambda_2)\left[\overline{\varphi}+(n-1)\underline{\eta}\right]}{\sqrt{n\left[\overline{\varphi}^2+(n-1)\underline{\eta}^2\right]}}$ and $K_2=\overline{\varphi}(1-\lambda_n)$.
%\mathcal{S}_{\alpha}=\left\{(\alpha,\beta)|\alpha<1-\lambda_n, \beta<\frac{1}{\overline{\varphi}(1-\lambda_n)}, \\\alpha\beta\overline{\varphi}(1-\lambda_n)<(1-s_1)(1-s_2)\right\}.
Consider the sequence $\{\bm{x}(k)\}_{k=0}^{\infty}$ and $\{\bm{y}(k)\}_{k=0}^{\infty}$ generated by \eqref{ald1} and \eqref{ald2}. Suppose Assumptions \ref{lip}-\ref{noi} hold. For any $\bm{x}(0)\in\mathbb{R}^{n}$ and $(\alpha,\beta)$ satisfying \eqref{nafarange},
there exist $m_{1d}, m_{2d}>0$ and $q_d\in(0,1)$ such that for any $k\geq0$,
\begin{align}
%&\Vert\bm{x}(k+1)-\bm{x}(k)\Vert\leq m_1q^k\label{res1};
&\Vert\bm{y}(k)\Vert_E\leq m_{1d}q_d^k\label{nnres2},
\\&\Vert\nabla\check{\bm{f}}(\bm{x}(k))\Vert_E\leq m_{2d}q_d^k\label{nnres3}.
\end{align}
%where $\nabla\check{\bm{f}}(\bm{x}(k))= \bm{L}\nabla\bm{f}(\bm{x}(k))$.
\end{thm}
\begin{proof}
Similar to the proof of Theorem \ref{nconsto}, we have 
%\bar{\bm{x}}(k+1)-\bar{\bm{d}}=(1-\alpha)(\bar{\bm{x}}(k)-\bar{\bm{d}})+\bar{\bm{\zeta}}(k).\notag
%Hence, 
\begin{align}
\Vert\bar{\bm{y}}(k)\Vert_E&=\Vert\bar{\bm{x}}(k+1)-\bar{\bm{d}}\Vert_E\notag
%\\&\leq (1-\alpha)^k\Vert\bar{\bm{x}}(0)-\bar{\bm{d}}\Vert_E+\sum_{t=0}^{k-1}(1-\alpha)^{k-1-t}m_{\zeta}q_{\zeta}^{t}\notag
\\&\leq(1-\alpha)^k\Vert\bar{\bm{x}}(0)-\bar{\bm{d}}\Vert_E+m_{\zeta}\frac{(1-\alpha)^{k}-q_{\zeta}^{k}}{1-\alpha-q_{\zeta}}.
\notag
\end{align}
We can see that for any $m_{0d}>\Vert\bar{\bm{x}}(0)-\bar{\bm{d}}\Vert_E+\frac{m_{\zeta}}{\vert1-\alpha-q_{\zeta}\vert}$ and $q_{0d}\in\big(\max\{\vert1-\alpha\vert, q_{\zeta}\},1\big)$ such that 
\begin{align}
\Vert\bar{\bm{y}}(k)\Vert_E\leq m_{0d}q_{0d}^k,\label{nnbarcon}
\end{align}
and we only need to prove that there exist $m_{1d}', m_{2d}'>0$ and $q_d'\in(q_0,1)$ such that
\begin{align}
%&\Vert\check{\bm{x}}(k+1)-\check{\bm{x}}(k)\Vert\leq m_{1d}'q_d'^k\label{res11};
&\Vert\check{\bm{y}}(k)\Vert_E\leq m_{1d}'q_d'^k\label{nnres21},
\\&\Vert\nabla\check{\bm{f}}(\bm{x}(k))\Vert_E\leq m_{2d}'q_d'^k\label{nnres31},
\end{align}

Similar to \eqref{th1e5} and \eqref{th1ex7}, we still have
\begin{align}
\nabla\check{\bm{f}}(\bm{x}(k+1))=&\left(\bm{I}-\beta\bm{L}\bm{\Theta}(k)(\bm{I}-\bm{W}(k))\right)\nabla\check{\bm{f}}(\bm{x}(k))\notag
\\&-\alpha\bm{L}\bm{\Theta}(k)\check{\bm{y}}(k)-\alpha\bm{L}\bm{\Theta}(k)\bar{\bm{y}}(k)\notag
\\&+\bm{L}\bm{\Theta}(k)\bm{\zeta}(k).\label{nnth1e5}
\end{align}
\begin{align}
\check{\bm{y}}(k+1)=&(\bm{W}(k)-\alpha\bm{I})\check{\bm{y}}(k)+\check{\bm{\zeta}}(k)\notag
\\&-\beta(\bm{I}-\bm{W}(k))\nabla\check{\bm{f}}(\bm{x}(k)).
\label{nnth1ex7}
\end{align}

Next, we will prove \eqref{nnres21} and \eqref{nnres31} by mathematical induction based on \eqref{nnth1e5} and \eqref{nnth1ex7}.

It is not difficult to prove that when $k=0$, there exist $m_{1d}', m_{2d}'>0$ and $q_d'\in(\max\{q_{0d},q_{\zeta}\},1)$ such that \eqref{nnres21} and \eqref{nnres31} hold.

Then, we assume that for all $k\leq\kappa$, \eqref{nnres21} and \eqref{nnres31} hold.

When $k=\kappa+1$,
from \eqref{nnth1ex7}, because $\bm{W}(k)$ is independent of $\bm{x}(k)$ and $\bm{y}(k)$, 
we have 
\iffalse
\begin{align}
\Vert\check{\bm{x}}(k+1)-\check{\bm{x}}(k)\Vert_E&\leq\alpha\Vert\check{\bm{y}}(k)\Vert_E+\beta(1-\lambda_n)\Vert\nabla\check{\bm{f}}(\bm{x}(k))\Vert_E\notag
\\&\leq \alpha m_{1d}'q_d'^k+\beta(1-\lambda_n)m_{2d}'q_d'^k.\label{th1e8}
%\\\Vert\check{\bm{y}}(k+1)\Vert\leq\lambda_n\Vert\check{\bm{y}}(k)\Vert+\Vert\check{\bm{x}}(k+1)-\check{\bm{x}}(k)\Vert\notag
%\\\Vert\nabla\check{\bm{f}}(\bm{x}(k+1))\Vert\leq s\Vert\nabla\check{\bm{f}}(\bm{x}(k))\Vert+\alpha\eta_{\max}\Vert\check{\bm{y}}(k)\Vert
%\Vert\left(\bm{I}-\beta\bm{L}\Theta(k)(\bm{I}-\bm{W})\right)\Vert
\end{align}
\fi
%Based on \eqref{th1e6} and \eqref{th1e8}, we have 
\begin{align}
\Vert\check{\bm{y}}(k+1)\Vert_E&\leq s_1\Vert\check{\bm{y}}(k)\Vert_E+\beta(1-\underline{\lambda}_n)\Vert\nabla\check{\bm{f}}(\bm{x}(k))\Vert_E+m_{\zeta}q_{\zeta}^k\notag
\\&\leq \left[s_1m_{1d}'+\beta(1-\underline{\lambda}_n)m_{2d}'+m_{\zeta}\right]q_d'^k.\label{nth1ex8}
\end{align}
%where $s_1'=\max\{\lambda_2-\alpha, \alpha-\underline{\lambda}_n\}$.
From \eqref{nnth1e5} and \eqref{nnbarcon}, we obtain based on Lemma \ref{lem1} that 
\begin{align}
\Vert\nabla\check{\bm{f}}(\bm{x}(k+1))\Vert_E\leq &s_2\Vert\nabla\check{\bm{f}}(\bm{x}(k))\Vert_E+\overline{\varphi}\Vert\bm{\zeta}(k)\Vert_E\notag
\\&+\alpha\overline{\varphi}(\Vert\check{\bm{y}}(k)\Vert_E+\Vert\bar{\bm{y}}(k)\Vert_E)\notag
%\\\leq &s_2'm_{2d}'q_d'^k+\alpha\overline{\varphi}m_{1d}'q_d'^k+\alpha\overline{\varphi}\Vert\bar{\bm{y}}(k)\Vert_E\notag
\\\leq &\left[s_2m_{2d}'+\alpha\overline{\varphi}(m_{1d}'+m_0)+\overline{\varphi}m_{\zeta}\right]q_d'^k.\label{nnth1ex9}
\end{align}
%From Lemma 1, we know that $s_2=\max\{1-\beta K_1, \beta K_2'-1\}$, where $K_1=\frac{\underline{\eta}(1-\lambda_2)\left[\overline{\varphi}+(n-1)\underline{\eta}\right]}{\sqrt{n\left[\overline{\varphi}^2+(n-1)\underline{\eta}^2\right]}}$ and $K_2=\overline{\varphi}(1-\underline{\lambda}_n)$.
%where $s_2=\max_k\{\rho\left(\bm{I}-\beta\bm{L}\bm{\Theta}(k)(\bm{I}-\bm{W})\right)\}$. From Lemma 1, we know that $s>0$ and $s$ does not depend on $k$.
%where $q_d'\geq1-\alpha$.
%Due to $\alpha<1-\underline{\lambda}_n, \ \beta<\frac{2K_1}{K_2^2}$, we have $s_1<1$ and $s_2<1$.
%Since
%&\alpha<1-\lambda_n, \ \beta<\frac{2K_1}{K_2'^2},
%$\alpha\beta\overline{\varphi}(1-\underline{\lambda}_n)<(1-s_1')(1-s_2')$,
%there exist $m_{1d}'>0$ and $m_{2d}'>0$ such that 
%\begin{align}
%&\alpha\beta\overline{\varphi}(1-\underline{\lambda}_n)(m_{1d}'+m_{0d})+\beta(1-\underline{\lambda}_n)\overline{\varphi}m_{\zeta}\notag
%\\&<\beta(1-\underline{\lambda}_n)(1-s_2')m_{2d}'\notag
%\\&< (1-s_1')(1-s_2')m_{1d}'-(1-s_2')m_{\zeta}.\label{ss1}
%\end{align}
Similar to \eqref{mth1ex10} and \eqref{mth1ex11}, we have there exist $m_{1d}'>0$, $m_{2d}'$ and $q_d'\in(\max\{q_{0d},q_{\zeta}\}, 1)$ such that
\begin{align}
\left[s_1m_{1d}'+\beta(1-\underline{\lambda}_n)m_{2d}'+m_{\zeta}\right]q_d'^k\leq m_{1d}'q_d'^{k+1},\label{nth1ex10}
\\\left[s_2m_{2d}'+\alpha\overline{\varphi}(m_{1d}'+m_{0d})+\overline{\varphi}m_{\zeta}\right]q_d'^k\leq m_{2d}'q_d'^{k+1},\label{nth1ex11}
\end{align}
Substituting \eqref{nth1ex10} and \eqref{nth1ex11} into \eqref{nth1ex8} and \eqref{nnth1ex9} yields
\begin{align}
&\Vert\check{\bm{y}}(k+1)\Vert_E\leq m_{1d}'q_d'^{k+1},\notag
\\&\Vert\nabla\check{\bm{f}}(\bm{x}(k+1))\Vert_E\leq m_{2d}'q_d'^{k+1}.\notag
\end{align}
This completes the induction.
\iffalse
Then, based on \eqref{nth1ex8} and \eqref{nth1ex9}, we only need to prove that there exist $m_{1d}', m_{2d}'>0$ and $q_d'\in(1-\alpha,1)$ such that 
\begin{align}
\left[s_1'm_{1d}'+\beta(1-\underline{\lambda}_n)m_{2d}'+m_{\zeta}\right]q_d'^k\leq m_{1d}'q_d'^{k+1},\label{nth1ex10}
\\\left[s_2'm_{2d}'+\alpha\overline{\varphi}(m_{1d}'+m_0)+\overline{\varphi}m_{\zeta}\right]q_d'^k\leq m_{2d}'q_d'^{k+1},\label{nth1ex11}
\end{align}
which is equivalent to showing that there exist $m_{1d}', m_{2d}'>0$ such that
\begin{align}
s_1'm_{1d}'+\beta(1-\underline{\lambda}_n)m_{2d}'+m_{\zeta}< m_{1d}',\label{th1e9}
\\s_2'm_{2d}'+\alpha\overline{\varphi}(m_{1d}'+m_0)+\overline{\varphi}m_{\zeta}< m_{2d}'.\label{th1e10}
\end{align}
Combining \eqref{th1e9} and \eqref{th1e10}, we have
\begin{align}
&\alpha\beta\overline{\varphi}(1-\underline{\lambda}_n)(m_{1d}'+m_0)+\beta(1-\underline{\lambda}_n)\overline{\varphi}m_{\zeta}
\\&<\beta(1-\underline{\lambda}_n)(1-s_2')m_{2d}'\notag
\\&< (1-s_1')(1-s_2')m_{1d}'-(1-s_2')m_{\zeta}.\notag
\end{align}
If we make $m_{1d}'$ and $m_{2d}'$ large enough, $m_0$ and $m_{\zeta}$ will be neglected. We only need to ensure that 
\begin{align}
\alpha\beta\overline{\varphi}(1-\underline{\lambda}_n)<(1-s_1')(1-s_2').\label{th1e11}
\end{align}
Additionally, from \eqref{th1e11}, we know that $s_1'$ and $s_2'$ should be less than $1$. Therefore, $\alpha$ and $\beta$ should also satisfy 
\begin{align}
\alpha<1-\underline{\lambda}_n; \beta<\frac{1}{\overline{\varphi}(1-\underline{\lambda}_n)},
\end{align}
completing the induction.\fi

Then, based on \eqref{nnbarcon} and \eqref{nnres21}, we obtain that 
\begin{align}
\Vert\bm{y}(k)\Vert_E\leq\Vert\bar{\bm{y}}(k)\Vert_E+\Vert\check{\bm{y}}(k)\Vert_E\leq(m_{0d}+m_{1d}')q_d'^k,
\end{align}
which validates \eqref{nnres2}. The proof is completed.
\end{proof}
\begin{rem}
In Theorem 2, we consider decaying disturbance to validate the resilience of the proposed algorithm. If the noise does not decay to 0 but is bounded, i.e., $q_{\zeta}=1$, $\bm{x}(k)$ cannot converge to the optimal solution but the gap between $\bm{x}(k)$ and the optimal solution is bounded, which can be obtained from Theorem 2 by setting $q_{d}=1$.
\end{rem}
%and we obtain the range of $\alpha$, i.e.,
%\begin{align}
%\alpha<\frac{(1-s)(1-\lambda_n)}{\beta\overline{\varphi}(1-\lambda_n)+(1-s)},
%\end{align}
\iffalse
where
\begin{align}
1-s=\max\left\{\frac{\beta\underline{\eta}(1-\lambda_2)\left[\overline{\varphi}+(n-1)\underline{\eta}\right]}{\sqrt{n\left[\overline{\varphi}^2+(n-1)\underline{\eta}^2\right]}}, 2-\beta\overline{\varphi}(1-\lambda_n)\right\}
\end{align}
\fi

\begin{cor}
Suppose Assumptions 1-4 hold.
For any $\bm{x}(0)\in\mathbb{R}^{n}$ and $(\alpha,\beta)$ satisfying \eqref{nafarange},
the sequence $\{\bm{x}(k)\}_{k=0}^{\infty}$ generated by \eqref{ald1} and \eqref{ald2} converges linearly to $\bm{x}^*$ in mean square. The convergence rate is 
\begin{align}
\max\Big\{&s_1+\frac{\alpha\beta\overline{\varphi}(1-\underline{\lambda}_n)}{1-s_2}, s_2+\frac{\alpha\beta\overline{\varphi}(1-\underline{\lambda}_n)}{1-s_1}, \vert1-\alpha\vert, q_{\zeta}\Big\}.\notag%\label{lowerbounds}
\end{align} 
\end{cor}

Corollary 2 verifies that the proposed deviation-tracking method is able to eliminate the deviation $\bm{1}^T\bm{x}(k)-\bm{1}^T\bm{d}$ no matter whether it is caused by infeasible initialization or random disturbance. 
The reason why it is disturbance-resilient is that we adopt deviation feedback as shown in Fig. 1. The closed loop makes it more stable and exact than the weighted gradient method and dual splitting method, which adopt open loop to deal with the deviation $\bm{1}^T\bm{x}(k)-\bm{1}^T\bm{d}$ \cite{46}. 

It should be noted that both Corollary 1 and Corollary 2 shares the upper bound of stepsizes and the upper bound is independent of the parameter of disturbance, 
which shows the universal resilience property of distributed deviation-tracking algorithm. 

\subsection{Improving convergence rate}
In this section, we will analyze the optimal stepsize for fastest convergence rate over fixed communication networks.
The result can also be applied to stochastic networks, which will be validated in simulations.

%i.e. to achieve the smallest value of $q$ in Theorem \ref{fixedbounded}. 
\begin{thm}\label{fixedbounded}
Let $s_1'=\max\{\bar{\lambda}_{2e}-\alpha, \alpha-\underline{\lambda}_{ne}\}$ and $s_2'=\max\{|1-\beta K_1'|, |\beta K_2'-1|\}$, 
where $K_1'=\frac{\underline{\eta}(1-\bar{\lambda}_{2e})\left[\overline{\varphi}+(n-1)\underline{\eta}\right]}{\sqrt{n\left[\overline{\varphi}^2+(n-1)\underline{\eta}^2\right]}}$ and $K_2'=\overline{\varphi}(1-\underline{\lambda}_{ne})$.
%\mathcal{S}_{\alpha}=\left\{(\alpha,\beta)|\alpha<1-\underline{\lambda}_{ne}, \beta<\frac{1}{\overline{\varphi}(1-\underline{\lambda}_{ne})}, \\\alpha\beta\overline{\varphi}(1-\underline{\lambda}_{ne})<(1-s_1)(1-s_2)\right\}.
Consider the sequence $\{\bm{x}(k)\}_{k=0}^{\infty}$ and $\{\bm{y}(k)\}_{k=0}^{\infty}$ generated by \eqref{ald1} and \eqref{ald2}. Suppose Assumptions 1-4 hold and $\bm{W}(k)=\bm{W}, \forall k$. For any $\bm{x}(0)\in\mathbb{R}^{n}$ and $(\alpha,\beta)$ satisfying
\begin{equation}
\begin{split}
&\alpha<1-\underline{\lambda}_{ne}, \ \beta<\frac{1}{K_2'},
\\&\alpha\beta\overline{\varphi}(1-\underline{\lambda}_{ne})<(1-s_1')(1-s_2'),
\end{split}
\label{afarange}
\end{equation}
there exist $m_1, m_2>0$ and $q\in(0,1)$ such that for any $k\geq0$,
\begin{align}
%&\Vert\bm{x}(k+1)-\bm{x}(k)\Vert\leq m_1q^k\label{res1};
&\Vert\bm{y}(k)\Vert_E\leq m_1q^k\label{res2},
\\&\Vert\nabla\check{\bm{f}}(\bm{x}(k))\Vert_E\leq m_2q^k\label{res3}.
\end{align}

\end{thm}
The proof is shown in Appendix D.

\begin{thm}\label{optimalstepsize}
Suppose Assumptions 1-3 hold. The optimal stepsizes of $\alpha$ and $\beta$ are expressed as
\begin{align}
&\beta_{op}=\frac{2}{K_1'+K_2'},\label{opbeta}
\\&\alpha_{op}=\min\Bigg\{\frac{K_1'(1-\bar{\lambda}_{2e})}{K_2'}, \frac{K_1'(1+\underline{\lambda}_{ne})}{2K_1'+K_2'},\notag
\\&\frac{1}{2}\Big(1+\bar{\lambda}_{2e}-2\beta_{op} K_2'\notag
\\&+\sqrt{4(\beta_{op} K_2'-1)^2+(1-\bar{\lambda}_{2e})(5-\bar{\lambda}_{2e})}\Big),\notag
\\&\frac{1}{2}\left(3+\underline{\lambda}_{ne}-\sqrt{(3+\underline{\lambda}_{ne})^2-4\beta_{op}K_1'(1+\underline{\lambda}_{ne})}\right) \Bigg\}.\label{opafa}
\end{align}
%that is, when $\alpha=\alpha_{op}$ and $\beta=\beta_{op}$, the proposed algorithm convergence with fastest rate.
\end{thm}
\begin{proof}
\iffalse
From \eqref{th1ex10} and \eqref{th1ex11}, we obtain that 
\begin{align}
q\geq s_1+\beta(1-\underline{\lambda}_{ne})\frac{m_2'}{m_1'}, \ q\geq s_2+\alpha\overline{\varphi}\frac{m_1'}{m_2'}.\label{th2e1}
\end{align}
From \eqref{th1n1}, we have 
\begin{align}
\frac{\alpha\overline{\varphi}}{1-s_2}<\frac{m_2'}{m_1'}<\frac{1-s_1}{\beta(1-\underline{\lambda}_{ne})}.\label{th2e3}
\end{align}
Since that $m_1'$ and $m_2'$ can be any positive value as long as they satisfy \eqref{th2e3}, combining \eqref{th2e1}-\eqref{th2e3}, we have 
\begin{align}
q\geq\max\left\{s_1+\frac{\alpha\beta\overline{\varphi}(1-\underline{\lambda}_{ne})}{1-s_2}, s_2+\frac{\alpha\beta\overline{\varphi}(1-\underline{\lambda}_{ne})}{1-s_1}\right\}.
\end{align}
Since $q$ also needs to be larger than $\max\{\vert1-\alpha\vert,q_{\zeta}\}$, we have  
\fi
Similar to Corollary 2, the convergence rate is 
\begin{align}
\max\Big\{&s_1'+\frac{\alpha\beta\overline{\varphi}(1-\underline{\lambda}_{ne})}{1-s_2'}, s_2'+\frac{\alpha\beta\overline{\varphi}(1-\underline{\lambda}_{ne})}{1-s_1'}, \vert1-\alpha\vert, q_{\zeta}\Big\}.\label{lowerbound}
\end{align}
%Then, we only need to minimize the lower bound of $q$ shown in \eqref{lowerbound}.

First, we consider the optimal value of $\beta$.

Let $g_1(\alpha, \beta)=s_1'+\frac{\alpha\beta\overline{\varphi}(1-\underline{\lambda}_{ne})}{1-s_2'}$ and $g_2(\alpha, \beta)=s_2'+\frac{\alpha\beta\overline{\varphi}(1-\underline{\lambda}_{ne})}{1-s_1'}$. If $\beta>\frac{2}{K_1'+K_2'}$, we obtain $s_2'=\beta K_2'-1$. Then, we have 
\begin{align}
g_1(\alpha,\beta)&=s_1'+\frac{\alpha\beta\overline{\varphi}(1-\underline{\lambda}_{ne})}{2-\beta K_2'}, \notag
\\g_2(\alpha,\beta)&=\beta K_2'-1+\frac{\alpha\beta\overline{\varphi}(1-\underline{\lambda}_{ne})}{1-s_1'}.\notag
\end{align}
Both $g_1(\alpha, \beta)$ and $g_2(\alpha, \beta)$ are increasing with respect to $\beta$. Similarly, we can obtain that $g_1(\alpha,\beta)$ is independent of $\beta$ and $g_2(\alpha,\beta)$ is decreasing when $\beta\leq\frac{2}{K_1'+K_2'}$. Therefore, we obtain the best value of $\beta_{op}=\frac{2}{K_1'+K_2'}$.
%So we only need to consider that when $\beta\leq\frac{2}{K_1+K_2}$, in which case, $s_2=1-\beta K_1$ and 
%\begin{align}
%g_1(\alpha,\beta)&=s_1+\frac{\alpha\overline{\varphi}(1-\underline{\lambda}_{ne})}{K_1}, \notag
%\\g_2(\alpha,\beta)&=1-\beta K_1+\frac{\alpha\beta\overline{\varphi}(1-\underline{\lambda}_{ne})}{1-s_1}.\notag
%\end{align}
%We find that $g_1(\alpha,\beta)$ is independent of $\beta$. Due to \eqref{th1e11}, we have that when $s_2=1-\beta K_1$, $\frac{\alpha\overline{\varphi}(1-\underline{\lambda}_{ne})}{1-s_1}<1$, which means that $g_2(\alpha,\beta)$ is decreasing with respect to $\beta$.
%Therefore, we obtain the best value of $\beta_{op}=\frac{2}{K_1+K_2}$.

Next, we consider the optimal value of $\alpha$ when $\beta=\beta_{op}$.
If $\alpha\leq\frac{\bar{\lambda}_{2e}+\underline{\lambda}_{ne}}{2}$, we obtain $s_1'=\bar{\lambda}_{2e}-\alpha$ and 
\begin{align}
g_1(\alpha,\beta_{op})&=\bar{\lambda}_{2e}-\alpha+\frac{\alpha\overline{\varphi}(1-\underline{\lambda}_{ne})}{K_1'}, \notag
\\g_2(\alpha,\beta_{op})&=\frac{K_2'-K_1'}{K_1'+K_2'}+\frac{2K_1'}{K_1'+K_2'}\frac{\alpha\overline{\varphi}(1-\underline{\lambda}_{ne})}{1-\bar{\lambda}_{2e}+\alpha}.\notag
\end{align}
It is obvious that $g_2(\alpha,\beta)$ is increasing with respect to $\alpha$. Due to $\overline{\varphi}(1-\underline{\lambda}_{ne})=K_2'>K_1'$, $g_1(\alpha,\beta)$ is also increasing with respect to $\alpha$. 

When $\alpha>\frac{\bar{\lambda}_{2e}+\underline{\lambda}_{ne}}{2}$, we have $s_1'=\alpha-\underline{\lambda}_{ne}$,
\begin{align}
g_1(\alpha,\beta_{op})&=\alpha-\underline{\lambda}_{ne}+\frac{\alpha\overline{\varphi}(1-\underline{\lambda}_{ne})}{K_1'}, \notag
\\g_2(\alpha,\beta_{op})&=\frac{K_2'-K_1'}{K_1'+K_2'}+\frac{2K_1'}{K_1+K_2'}\frac{\alpha\overline{\varphi}(1-\underline{\lambda}_{ne})}{1+\underline{\lambda}_{ne}-\alpha}.\notag
\end{align}
Both of $g_1(\alpha,\beta_{op})$ and $g_2(\alpha, \beta_{op})$ are also increasing with respect to $\alpha$.
However, $1-\alpha$ is decreasing when $\alpha<1$. Then the optimal $\alpha$ should be
\begin{align}
\alpha_{op}=\min\{\alpha_{op1}, \alpha_{op2}\},\notag
\end{align}
where $\alpha_{op1}$ and $\alpha_{op2}$ satisfy
\begin{align}
1-\alpha_{op1}=g_1(\alpha_{op1}, \beta_{op}); 1-\alpha_{op2}=g_2(\alpha_{op2}, \beta_{op}).\notag
\end{align}
Finally, we calculate the optimal $\alpha$ and obtain \eqref{opafa}.
\end{proof}
The result in Theorem \ref{optimalstepsize} is obtained based on the convergence result of Theorem \ref{fixedbounded}. %The result is only optimal stepsize is only optimal for the result of Theorem \ref{fixedbounded}. 
In some cases, the result in Theorem \ref{optimalstepsize} may not be the optimal stepsize for fastest convergence rate, but it is still better than most stepsizes within the upperbound. 
%Although we obtain the optimal stepsizes based on fixed network,
The results are also suitable for stochastic networks with communication failure and disturbance, which will be shown in Section VII.
\iffalse
\begin{rem}
In fact, Theorem \ref{optimalstepsize} is established by the trade-off among the convergence rates of $\Vert\nabla\check{\bm{f}}(\bm{x}(k))\Vert_2$, $\Vert\check{\bm{y}}(k)\Vert_2$ and $\Vert\bar{\bm{y}}(k)\Vert_2$. If $\alpha$ is small enough, the proposed algorithm is similar to the weighted gradient method, where $\Vert\nabla\check{\bm{f}}(\bm{x}(k))\Vert_2$ converges to 0 fast but $\Vert\check{\bm{y}}(k)\Vert_2$ and $\Vert\bar{\bm{y}}(k)\Vert_2$ vanish slowly. If $\alpha=1$,  we can see that  $\Vert\bar{\bm{y}}(k)\Vert_2=0$ when $k\geq1$, but $\Vert\check{\bm{y}}(k)\Vert_2$ and $\Vert\nabla\check{\bm{f}}(\bm{x}(k))\Vert_2$ will converge slowly. 
\end{rem}
\fi

\section{Convergence with uncoordinated stepsizes}
In above analysis, all agents employ absolutely identical stepsizes. In distributed network, realizing uniform stepsizes requires consensus protocol, which is based on a reliable communication network. It is difficult to coordinate these stepsizes over stochastic networks. 
In this section, we analyze the convergence of the proposed algorithm with uncoordinated stepsizes. 

For each agent $i$, the stepsizes are denoted by $\alpha_i$ and $\beta_i$. Let $\overline{\alpha}=\max_i\{\alpha_i\}$, $\overline{\beta}=\max_i\{\beta_i\}$, $\underline{\alpha}=\min_i\{\alpha_i\}$, $\underline{\beta}=\min_i\{\beta_i\}$, $\bm{D}_{\alpha}=diag\{\alpha_1,...,\alpha_n\}$ and $\bm{D}_{\beta}=diag\{\beta_1,...,\beta_n\}$. Eqs. \eqref{al1} and \eqref{al2} are rewritten as 
\begin{align}
\bm{x}(k+1)=&\bm{x}(k)-\bm{D}_{\alpha}\bm{y}(k)-\bm{D}_{\beta}(\bm{I}-\bm{W}(k))\nabla\bm{f}(\bm{x}(k))\label{algg1},
\\\bm{y}(k+1)=&\bm{W}(k)\bm{y}(k)+\bm{x}(k+1)-\bm{x}(k)\label{algg2}.
\end{align}
We will prove \eqref{algg1} and \eqref{algg2} converge linearly to the optimal solution in mean square with uncoordinated stepsizes
\begin{lem}\label{unclem}
Let $K_1''=\frac{\underline{\beta}\,\underline{\eta}(1-\bar{\lambda}_{2e})\left[\overline{\beta}\,\overline{\varphi}+(n-1)\underline{\eta}\,\underline{\beta}\right]}{\sqrt{n\left[\overline{\beta}^2\overline{\varphi}^2+(n-1)\underline{\eta}^2\underline{\beta}^2\right]}}$, $K_2''=\overline{\beta}\,\overline{\varphi}(1-\underline{\lambda}_{ne})$ and $\bm{T}_d=\bm{L}\bm{\Gamma}\bm{D}_{\beta}(\bm{I}-\bm{W})$,
where $\bm{W}$ is a stochastic matrix satisfying Assumption 3. For any vector $\bm{v}\neq0$ such that $\bm{1}^T\bm{v}=0$, we have
\begin{align}
\Vert(\bm{I}-\bm{T}_d)\bm{v}\Vert_E^2\leq (1+K_2''^2-2K_1'')\Vert\bm{v}\Vert_E^2.\notag
\end{align}
\end{lem}
Lemma \ref{unclem} is an extension result of Lemma \ref{lem1} with replacing diagonal matrix $\bm{\Gamma}$ by another diagonal matrix $\bm{\Gamma}\bm{D}_{\beta}$.
\begin{thm}
Let $s_4=1-\sum_{i=1}^{n}\alpha_i$, $s_5=\overline{\alpha}^2+2\overline{\alpha}+\overline{\lambda}_{2s}$ and $s_6=\sqrt{1+K_2''^2-2K_1''}$.
%\mathcal{S}_{\alpha}=\left\{(\alpha,\beta)|\alpha<1-\underline{\lambda}_{ne}, \beta<\frac{1}{\overline{\varphi}(1-\underline{\lambda}_{ne})}, \\\alpha\beta\overline{\varphi}(1-\underline{\lambda}_{ne})<(1-s_1)(1-s_2)\right\}.
Consider the sequence $\{\bm{x}(k)\}_{k=0}^{\infty}$ and $\{\bm{y}(k)\}_{k=0}^{\infty}$ generated by \eqref{algg1} and \eqref{algg2}. Suppose Assumptions 1-4 hold. For any $\bm{x}(0)\in\mathbb{R}^{n}$ and $(\alpha,\beta)$ satisfying
\begin{align}
%\begin{split}
&\sum_{i=1}^{n}\alpha_i<2, \ \overline{\alpha}<\sqrt{2-\overline{\lambda}_{2s}}-1,\ K_2''^2<2K_1'',
\label{range1}
\\&(1-s_4)(1-s_5)(1-s_6)-\overline{\beta}(1-\underline{\lambda}_n)\overline{\varphi}\,\overline{\alpha}(2-s_4-s_5)\notag
\\&<\overline{\alpha}^2(1-s_6)+2\overline{\alpha}^2\overline{\beta}(1-\underline{\lambda}_n)\overline{\varphi},
%\end{split}
\label{range2}
\end{align}
there exist $m_{1u}, m_{2u}, m_{3u}>0$ and $q_{u}\in(0,1)$ such that for any $k\geq0$,
\begin{align}
%&\Vert\bm{x}(k+1)-\bm{x}(k)\Vert\leq m_1q^k\label{res1};
&\Vert\nabla\check{\bm{f}}(\bm{x}(k))\Vert_E\leq m_{1u}q_{u}^k,\label{res1}
\\&\Vert\bar{\bm{y}}(k)\Vert_E\leq m_{2u}q_{u}^k,\label{res2},
\\&\Vert\check{\bm{y}}(k)\Vert_E\leq m_{3u}q_{u}^k.\label{res3}
%\\&\Vert\nabla\check{\bm{f}}(\bm{x}(k))\Vert_E\leq m_2q^k\label{res3}.
\end{align}

\end{thm}
\begin{proof}
%\begin{align}
%\bm{x}(k+1)=&\bm{x}(k)-\bm{D}_{\alpha}\bm{y}(k)\notag
%\\&-\bm{D}_{\beta}(\bm{I}-\bm{W}(k))\nabla\bm{f}(\bm{x}(k))\label{algg1},
%\\\bm{y}(k+1)=&\bm{W}(k)\bm{y}(k)+\bm{x}(k+1)-\bm{x}(k)\label{algg2}.
%\end{align}
Multiplying both sides of \eqref{algg1} by $\frac{\bm{11}^T}{n}$, we have
\begin{align}
\bar{\bm{x}}(k+1)=&\bar{\bm{x}}(k)-\frac{\bm{11}^T}{n}\bm{D}_{\alpha}\bar{\bm{y}}(k)-\frac{\bm{11}^T}{n}\bm{D}_{\alpha}\check{\bm{y}}(k)\notag
\\&-\frac{\bm{11}^T}{n}\bm{D}_{\beta}(\bm{I}-\bm{W}(k))\nabla\check{\bm{f}}(\bm{x}(k)).\label{th5t1}
\end{align}
Then, multiplying both sides of \eqref{algg1} by $\bm{L}$ yields 
\begin{align}
\check{\bm{x}}(k+1)=&\check{\bm{x}}(k)-\bm{L}\bm{D}_{\alpha}\bar{\bm{y}}(k)-\bm{L}\bm{D}_{\alpha}\check{\bm{y}}(k)\notag
\\&-\bm{L}\bm{D}_{\beta}(\bm{I}-\bm{W}(k))\nabla\check{\bm{f}}(\bm{x}(k)).\label{th5t2}
\end{align}
Similarly, it follows from \eqref{algg2} that
\begin{align}
\bar{\bm{y}}(k+1)=&\bar{\bm{y}}(k)+\bar{\bm{x}}(k+1)-\bar{\bm{x}}(k),\label{th5t3}
\\\check{\bm{y}}(k+1)=&\bm{W}(k)\check{\bm{y}}(k)+\check{\bm{x}}(k+1)-\check{\bm{x}}(k).\label{th5t4}
\end{align}
Substituting \eqref{th5t1} and \eqref{th5t2} into \eqref{th5t3} and \eqref{th5t4} respectively yields
\begin{align}
\bar{\bm{y}}(k+1)=&(\bm{I}-\frac{\bm{11}^T}{n}\bm{D}_{\alpha})\bar{\bm{y}}(k)-\frac{\bm{11}^T}{n}\bm{D}_{\alpha}\check{\bm{y}}(k)\notag
\\&-\frac{\bm{11}^T}{n}\bm{D}_{\beta}(\bm{I}-\bm{W}(k))\nabla\check{\bm{f}}(\bm{x}(k)),\label{th5t5}
\\\check{\bm{y}}(k+1)=&(\bm{W}(k)-\bm{L}\bm{D}_{\alpha})\check{\bm{y}}(k)-\bm{L}\bm{D}_{\alpha}\bar{\bm{y}}(k)\notag
\\&-\bm{L}\bm{D}_{\beta}(\bm{I}-\bm{W}(k))\nabla\check{\bm{f}}(\bm{x}(k)).\label{th5t6}
\end{align}
Similar to \eqref{th1e4}, we have 
\begin{align}
\nabla\check{\bm{f}}(\bm{x}(k+1))=&\left(\bm{I}-\bm{L}\bm{\Theta}(k)\bm{D}_{\beta}(\bm{I}-\bm{W}(k))\right)\nabla\check{\bm{f}}(\bm{x}(k))\notag
\\&-\bm{L}\bm{\Theta}(k)\bm{D}_{\alpha}\check{\bm{y}}(k)-\bm{L}\bm{\Theta}(k)\bm{D}_{\alpha}\bar{\bm{y}}(k).\label{th5t7}
\end{align}

We use induction to prove \eqref{res1}-\eqref{res3}. 
When $k=0$, it is not difficult to find $m_{1u}, m_{2u}, m_{3u}>0$ and $q_{u}\in(0,1)$ such that \eqref{res1}-\eqref{res3} hold.
We assume when $k\leq\kappa$, \eqref{res1}-\eqref{res3} hold.
When $k=\kappa+1$, 
from \eqref{th5t5} and \eqref{th5t6}, we have 
\begin{align}
\Vert\bar{\bm{y}}(\kappa+1)\Vert_{E}\leq&s_4\Vert\bar{\bm{y}}(\kappa)\Vert_E+\overline{\alpha}\Vert\check{\bm{y}}(\kappa)\Vert_E\notag
\\&+\overline{\beta}(1-\underline{\lambda}_n)\Vert\nabla\check{\bm{f}}(\bm{x}(\kappa))\Vert_E\notag
\\\leq&s_4m_2q^{\kappa}+\overline{\alpha}m_3q^{\kappa}+\overline{\beta}(1-\underline{\lambda}_n)m_1q^{\kappa},\label{th5t11}
\\\Vert\check{\bm{y}}(\kappa+1)\Vert_E\leq&s_5\Vert\check{\bm{y}}(\kappa)\Vert_E+\overline{\alpha}\Vert\bar{\bm{y}}(\kappa)\Vert_E\notag
\\&+\overline{\beta}(1-\underline{\lambda}_n)\Vert\nabla\check{\bm{f}}(\bm{x}(\kappa))\Vert_E\notag
\\\leq&s_5m_3q^{\kappa}+\overline{\alpha}m_2q^{\kappa}+\overline{\beta}(1-\underline{\lambda}_n)m_1q^{\kappa}.\label{th5t12}
\end{align}
Then, it can be obtained from \eqref{th5t7} that 
\begin{align}
\Vert\nabla\check{\bm{f}}(\bm{x}(\kappa+1))\Vert_E\leq&s_6\Vert\nabla\check{\bm{f}}(\bm{x}(\kappa))\Vert_E\notag
\\&+\overline{\eta}\,\overline{\alpha}\Vert\check{\bm{y}}(\kappa)\Vert_E+\overline{\eta}\,\overline{\alpha}\Vert\bar{\bm{y}}(\kappa)\Vert_E\notag
\\\leq&s_6m_1q^{\kappa}+\overline{\eta}\,\overline{\alpha}m_2q^{\kappa}+\overline{\eta}\,\overline{\alpha}m_3q^{\kappa}.\label{th5t13}
\end{align}
Due to \eqref{range1}, we have $s_4$, $s_5$ and $s_6$ are all in $(0, 1)$. Then, it follows from \eqref{range2} that  
\begin{multline}
\frac{(1-s_4)(1-s_6)-\overline{\beta}(1-\underline{\lambda}_n)\overline{\varphi}\,\overline{\alpha}}{\overline{\alpha}(1-s_6)+\overline{\beta}(1-\underline{\lambda}_n)\overline{\varphi}\,\overline{\alpha}}
\\<\frac{\overline{\alpha}(1-s_6)+\overline{\beta}(1-\underline{\lambda}_n)\overline{\varphi}\,\overline{\alpha}}{((1-s_5)(1-s_6)-\overline{\beta}(1-\underline{\lambda}_n)\overline{\varphi}\,\overline{\alpha})}.\label{th5t10}
\end{multline}
Based on \eqref{th5t10},  there exist $m_2, m_3>0$ such that
\begin{multline}
(\overline{\alpha}(1-s_6)+\overline{\beta}(1-\underline{\lambda}_n)\overline{\varphi}\,\overline{\alpha})m_3\notag
\\<((1-s_4)(1-s_6)-\overline{\beta}(1-\underline{\lambda}_n)\overline{\varphi}\,\overline{\alpha})m_2,
\end{multline}
\begin{multline}
(\overline{\alpha}(1-s_6)+\overline{\beta}(1-\underline{\lambda}_n)\overline{\varphi}\,\overline{\alpha})m_3
\\<((1-s_5)(1-s_6)-\overline{\beta}(1-\underline{\lambda}_n)\overline{\varphi}\,\overline{\alpha})m_2.\notag
\end{multline}
Then, we obtain that there exist $m_1, m_2, m_3>0$ such that
\begin{align}
\overline{\alpha}m_3+\overline{\beta}(1-\underline{\lambda}_n)m_1<&(1-s_4)m_2,\label{th5t14}
\\\overline{\alpha}m_2+\overline{\beta}(1-\underline{\lambda}_n)m_1<&(1-s_5)m_3,\label{th5t15}
\\\overline{\eta}\  \overline{\alpha}m_2+\overline{\eta}\  \overline{\alpha}m_3<&(1-s_6)m_1.\label{th5t16}
\end{align}
Substituting \eqref{th5t14}-\eqref{th5t16} into \eqref{th5t11}-\eqref{th5t13} leads to \eqref{res1}-\eqref{res3} when $k=\kappa+1$. This completes the proof.
\end{proof}
From Theorem 5, we can deduce that the convergence point of $\bm{x}(k)$ satisfies the two conditions in Proposition \ref{optimalcondition}, which means that the proposed algorithm still converges linearly to the optimal solution in mean square even with uncoordinated stepsizes. Thus, agents do not need to coordinate stepsizes before implementing the algorithm.

\section{Numerical results}
In this section, simulations are provided to validate the effectiveness of the proposed deviation-tracking algorithm (DTA).
Considering a network of 10 agents, of which the topology is randomly generated. The cost functions are chosen to be quadratic functions: 
\begin{align}
f_i(x)=a_ix_i^2+b_ix_i+c_i, \forall i,\notag
\end{align}
which are widely used in economic dispatch problem to model the cost induced by producing certain amount of power \cite{51}. The values of $a$, $b$ and $c$ are randomly chosen within the ranges that are used in \cite{36}. Their values are shown as follows.
     
     \begin{center}  
       \begin{tabular}{|c|c|c|c|c|c|}
         %\hline
         % after \\: \hline or \cline{col1-col2} \cline{col3-col4} ...
        %Agent& a & b&c \\
 \hline
 Agent&1&2&3&4&5\\
 \hline
 $a_i$&0.0314& 0.0342 &0.0392 &0.0379 &0.0366 \\
 $b_i$&0.352&0.349&0.278&0.331& 0.234\\
 $c_i$&0&0&0&0&0\\
  \hline
  Agent&6&7&8&9&10\\
   \hline
 $a_i$&0.0304 &0.0385 & 0.0393 &0.0368 &0.0396\\
 $b_i$&0.341&0.206&0.255&0.209&0.219\\
 $c_i$&0&0&0&0&0\\
%       1&  0.0314 & 0.352&0 \\
%      2& 0.0342 & 0.349&0 \\
%      3&   0.0392 & 0.278&0 \\
%       4&  0.0379 & 0.331&0 \\
%       5&  0.0366 & 0.234&0 \\
%       6&  0.0304 & 0.341&0 \\
%      7&   0.0385 & 0.206&0 \\
%      8&   0.0393 & 0.255&0 \\
%      9&   0.0368 & 0.209&0 \\
%      10&   0.0396 & 0.219&0 \\
         \hline
       \end{tabular}
       %\centerline{Table 1}
  \end{center}  

\subsection{Convergence of deviation-tracking algorithm}

First, we simulate the convergence results of deviation-tracking algorithm with communication failure while no disturbance is injected. 
Each communication link fails with probability $1-\theta$, where $\theta=0.5$.
We select the stepsize $\alpha$ to be $0.1\alpha_{op}$, $0.3\alpha_{op}$, $0.5\alpha_{op}$, $0.7\alpha_{op}$, $0.8\alpha_{op}$, $0.9\alpha_{op}$, $0.95\alpha_{op}$ and $\alpha_{op}$, where $\alpha_{op}$ is calculated by \eqref{opafa}. The stepsizes are all constant and satisfy \eqref{nafarange}. We also show the convergence result when $\alpha$ equals the upper bound, which is obtained by \eqref{nafarange}. Another stepsize $\beta$ is fixed to be $\beta_{op}$. We use $\Vert\bm{x}(k)-\bm{x}^*\Vert_2$ to describe the distance between the state at time $k$ and the optimal solution. 
\begin{figure}
		\centering
		\includegraphics[scale=0.17]{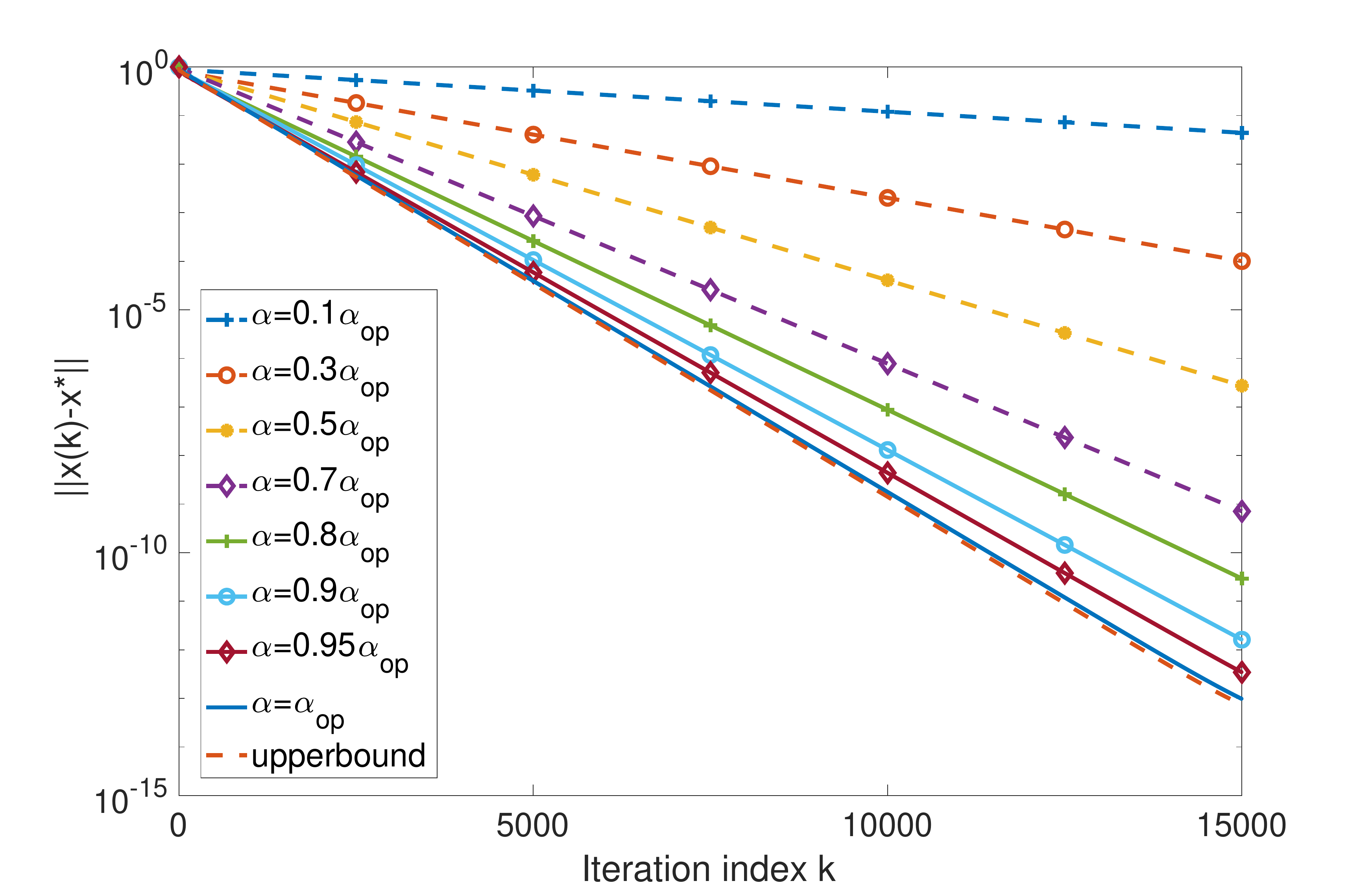}
\caption{Convergence results of DTA with different stepsizes $\alpha$.} 
\end{figure}
The results are shown in Fig. 2.

We obtain from Fig. 2 that the state converge to the optimal solution with all the selected constant stepsizes.  This confirms Theorem \ref{nconsto} and suggests that the proposed algorithm does not need decaying stepsize to ensure convergence. The advantage of adopting constant stepsizes is also shown in Fig. 2 that $\bm{x}(k)$ converges at a linear rate for all the chosen stepsizes. 
Moreover, it is worth noting that
when $\alpha\leq\alpha_{op}$, the convergence rate become faster when $\alpha$ increases and 
when $\alpha=\alpha_{op}$, the algorithm converges faster than that with other selected stepsizes except the upper bound. 
In fact, the optimal stepsize should be extremely close to the upper bound. The reason why there exists a small gap between them is that the theoretical result is obtained based on generic functions defined in Assumptions 1 and 2 instead of specific quadratic functions. It is, however, observed that the theoretical result is quite close to the optimal stepsize, which suggests that the result of Theorem \ref{optimalstepsize} is also suitable for stochastic networks.

To validate the optimal stepsize of $\beta$, i.e. $\beta_{op}$ calculated in Therorem \ref{optimalstepsize}, we select 14 different values of $\beta$, i.e., $0.1\beta_{op}$, $0.3\beta_{op}$, $0.5\beta_{op}$, $0.7\beta_{op}$, $0.8\beta_{op}$, $0.9\beta_{op}$, $\beta_{op}$, $1.01\beta_{op}$, $1.02\beta_{op}$, $1.04\beta_{op}$, $1.06\beta_{op}$, $1.07\beta_{op}$, $1.075\beta_{op}$, $1.08\beta_{op}$. 
The stepsize $\alpha$ is fixed as $\alpha_{op}$. The convergence rate is calulated by 
\begin{align}
q_{n}=\sqrt[N]{\prod_{i=k_e-N+1}^{k_e}\frac{\Vert\bm{x}(k)-\bm{x}_{op}\Vert_2}{\Vert\bm{x}(k-1)-\bm{x}_{op}\Vert_2}},\notag
\end{align}
where $k_e$ denotes the index of the last iteration and $N$ means the amount of the sample. In this experiment, we choose  $k_e=25000$ and $N=1000$.
The result is shown in Fig. 3. We find that the convergence rate $q_n$ decrease gradually first and then increase rapidly.  
The numerical optimal $\beta$ is quite close to the theoretical result $\beta_{op}$.
\begin{figure}
		\centering
		\includegraphics[scale=0.19]{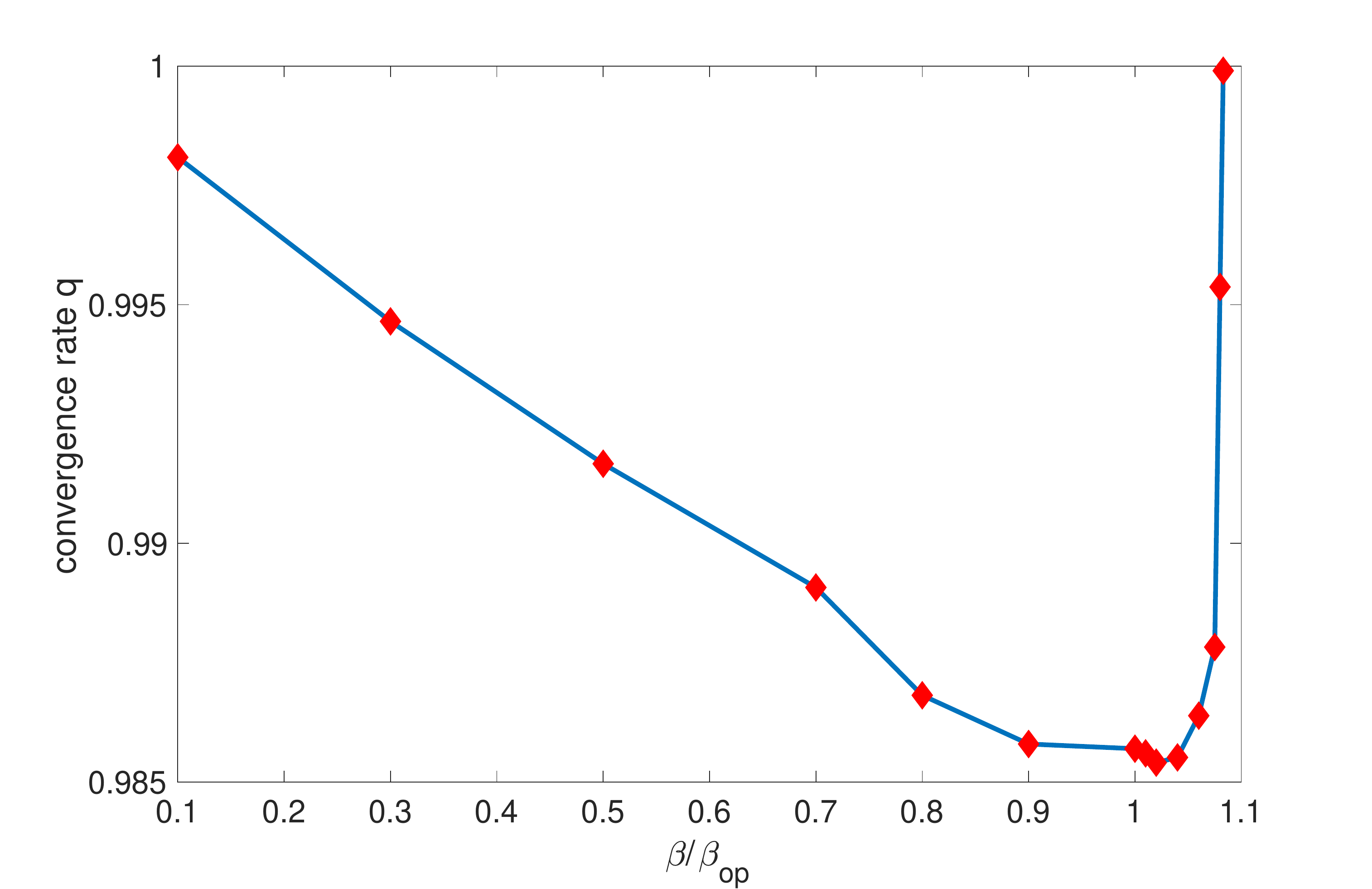}
\caption{Convergence rate $q$ versus the constanst stepsize $\beta$.} 
\end{figure}

To validate the convergence with uncoordinated stepsizes, we randomly generate stepsizes within the range defined by \eqref{range1} and \eqref{range2}. The maximum and minimum of these stepsizes are denoted by $\overline{\alpha}$ and $\underline{\alpha}$, respectively. We first simulate the convergence with these uncoordinated stepsizes. Then, we plot the convergence results when $\alpha_i=\overline{\alpha}, \forall i$ and $\alpha_i=\underline{\alpha}, \forall i$, respectively. The convergence results are shown in Figure 4.
\begin{figure}
		\centering
		\includegraphics[scale=0.22]{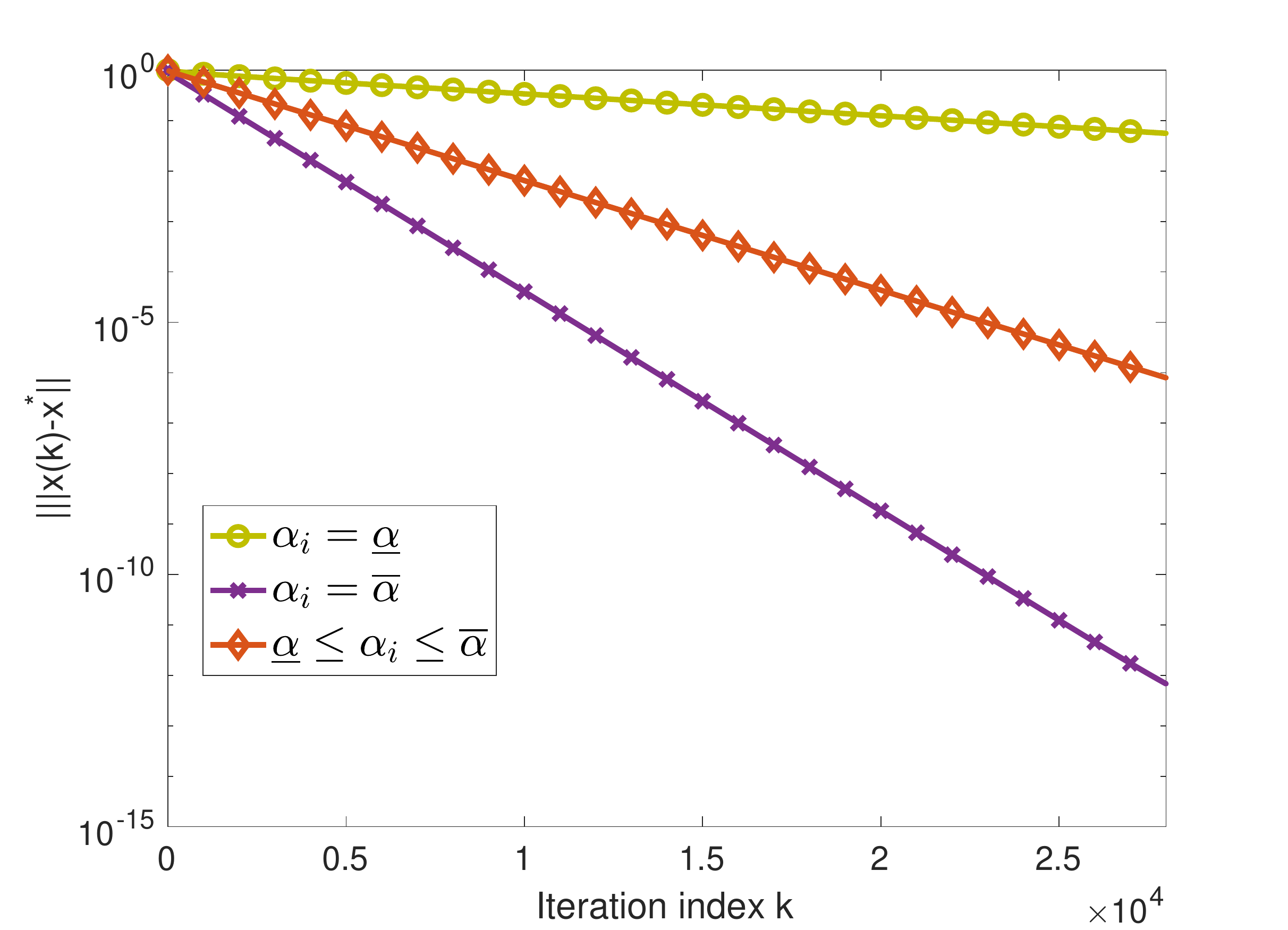}
\caption{Convergence results of DTA with uncoordinated stepsizes.} 
\end{figure}
The figure suggests that the proposed algorithm converges linearly to the optimal solution with uncoordinated stepsizes. We find that the states converge at a faster rate than that when the all agents' stepsizes equal to $\underline{\alpha}$, and slower than that when $\alpha_i=\overline{\alpha}$. So, if $\overline{\alpha}-\underline{\alpha}$ is small, we can  obtain an estimate of the convergence rate with uncoordinated stepsizes.

Next, we analyze the convergence results with disturbance injected in states. The disturbances we choose to inject to the states are Gaussian. To meet the requirement in Assumption \ref{noi}, we set the variance of $\zeta_i(k), \forall i$ converges to 0 linearly.
%$\bm{\zeta}(k)$ to satisfy
%\begin{align}
%\bm{\zeta}(k)\sim Gau_n(0, \sigma_k^2),
%\end{align}
%where $\sigma_k=$
We select the same eight values as above for stepsize $\alpha$ and set $\beta=\beta_{op}$. 
The results are shown in Fig. 5. 
We can see that states still converge to the optimal solutions even with disturbance. With all the selected stepsizes, the algorithm still converges at a linear rate. 
Moreover, we find that when $\alpha\leq0.7\alpha_{op}$, the convergence rate is faster as $\alpha$ grows. But $\alpha>0.7\alpha_{op}$, the convergence rate does not change obvious and greatly affected by the disturbance. 
\begin{figure}
		\centering
		\includegraphics[scale=0.165]{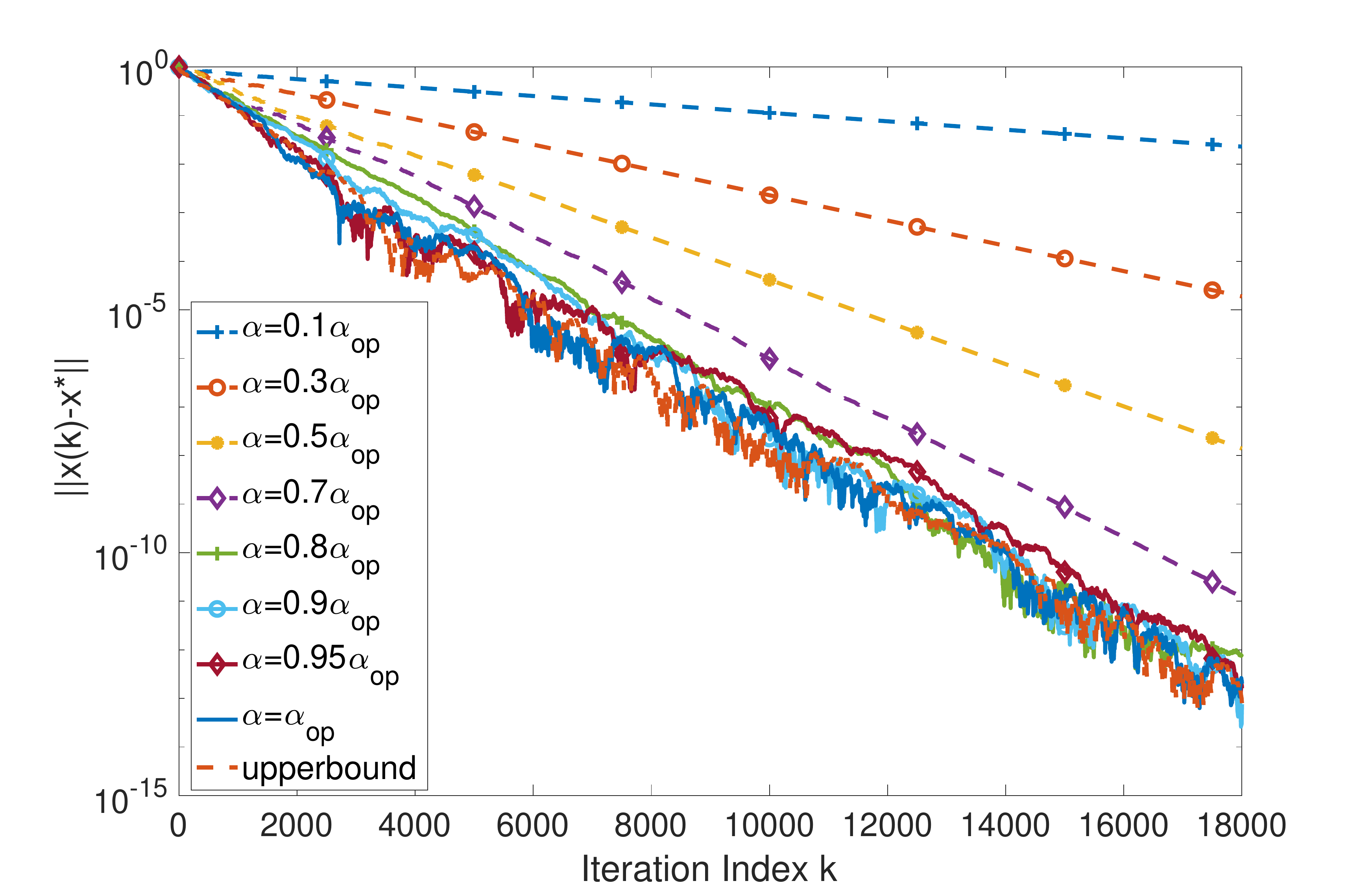}
\caption{Convergence results of DTA with different stepsizes $\alpha$ under disturbance.} 
\end{figure}
This is because the convergence rate is also influenced by the decaying rate of the disturbance, i.e. $q_{\zeta}$ in \eqref{lowerbound}, which is independent of $\alpha$. 

To find how the probability of failure influence convergence rate. We select 10 different value of $\theta$, i.e., 0.9, 0.1, 0.05, 0.03, 0.025, 0.02, 0.015, 0.01, 0.05 and 0. The results are shown Fig. 6. We can see that when $\theta>0.05$, the convergence rate does not change. The reason is that $\vert1-\alpha\vert>\max\Big\{s_1'+\frac{\alpha\beta\overline{\varphi}(1-\underline{\lambda}_n)}{1-s_2'}, s_2'+\frac{\alpha\beta\overline{\varphi}(1-\underline{\lambda}_n)}{1-s_1'}\Big\}$ and $q_s'=|1-\alpha|$, where $\vert1-\alpha\vert$ is independent of $\theta$ while $s_{2}'$ increases as $\theta$ grows. When $\theta\leq0.03$, we find the convergence rate reduces when $\theta$ increases. Especially, when $\theta=0$ the algorithm cannot converge to the optimal solution because the network becomes a fixed disconnected network, which does not satisfy the requirement of Assumption 3.
\begin{figure}
		\centering
		\includegraphics[scale=0.16]{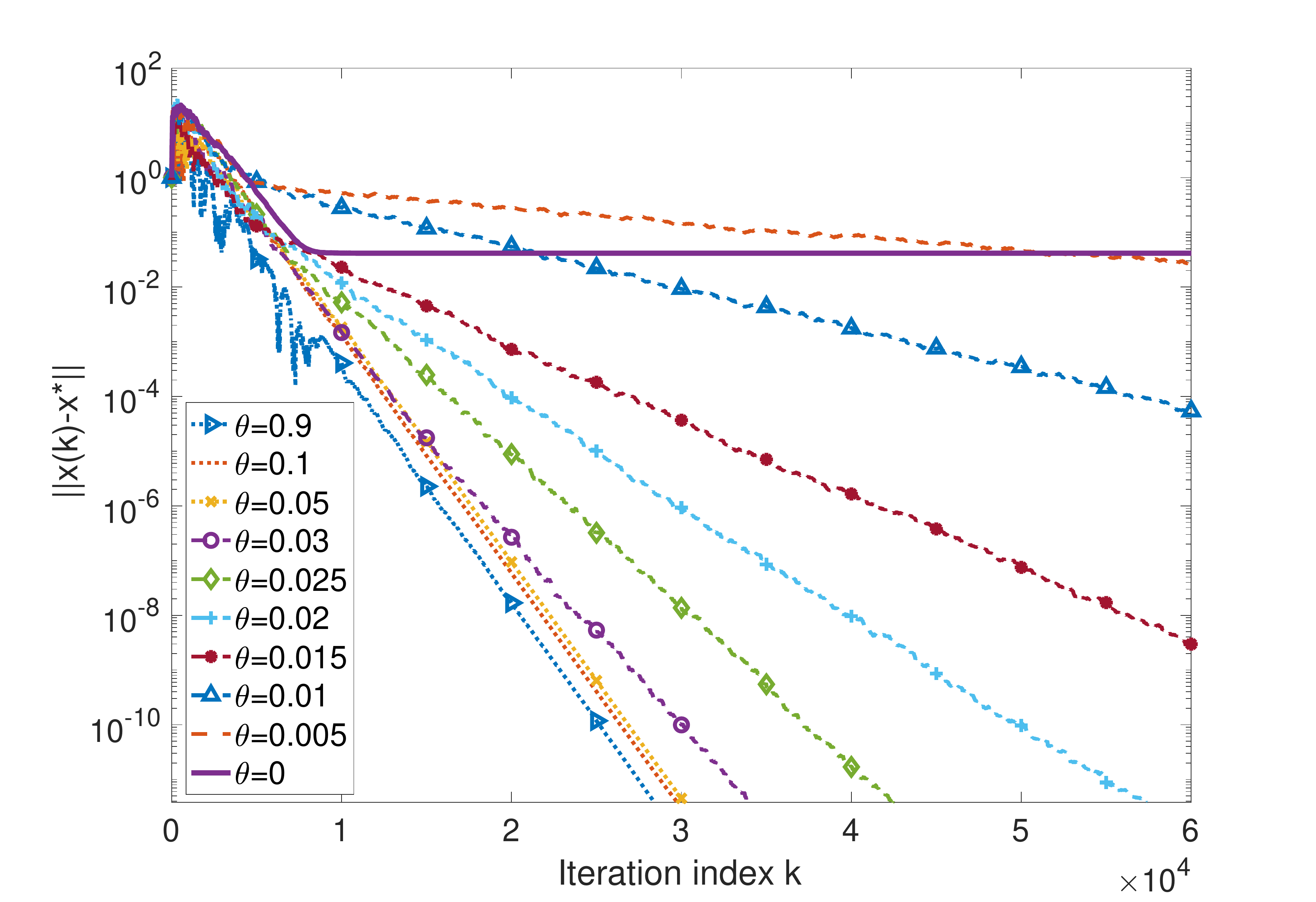}
\caption{Convergence results of DTA with different communication failure probabilities.} 
\end{figure}

\subsection{Comparison with the state-of-art}
Finally, we compare the proposed algorithm with weighted gradients algorithm (WGA) \cite{46} and dual splitting algorithm (DuSPA) \cite{36}.
We consider two types of disturbance: Gaussian disturbance and Laplace disturbance. To meet the requirement of Assumption \ref{noi}, we set the mean and the variance of disturbance decays to 0 linearly. We add these kinds of disturbance to DTA, WGA and DuSPA and obtain the results shown in Fig. 7. 
\begin{figure}
		\centering
		\includegraphics[scale=0.165]{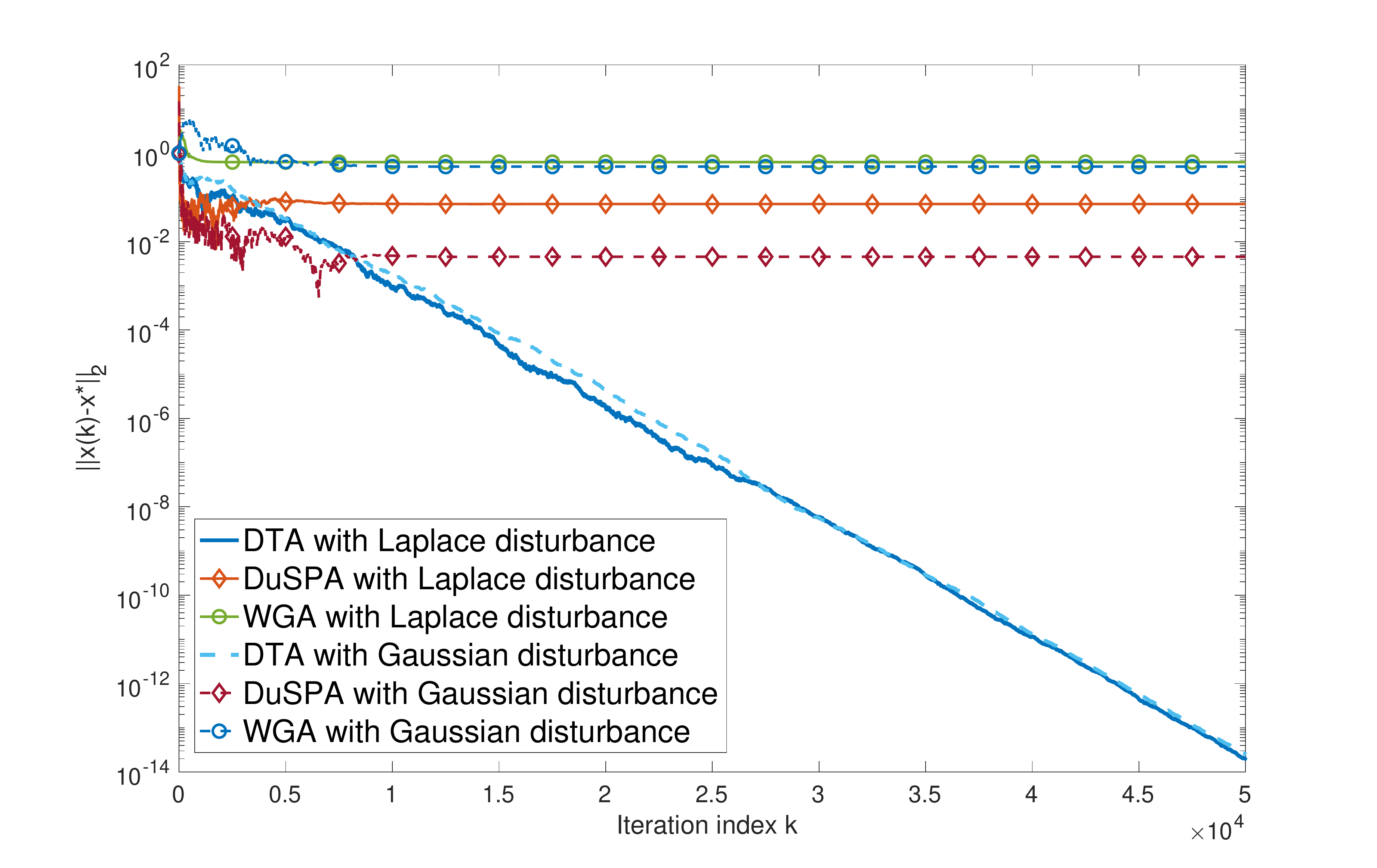}
\caption{Convergence results of DTA, WGA and DuSPA in the presence of Laplace and Gaussian disturbance, respectively.
} 
\end{figure}
We can see that 
%the proposed algorithm still converges to the optimal solution 
%with constant stepsizes and do not require decaying stepsizes to ensure convergence. Moreover, the advantage of adopting constant stepsizes is also shown in Fig. 1 that state $\bm{x}(k)$ converges to the optimal point at a linear rate. 
utilizing the deviation-tracking method, DTA still converges linearly to the optimal solution even under disturbance while WGA and DuSPA cannot converge to the optimal solution. 
%The numerical results in Fig. 6 corroborate Theorem \ref{consto} and Proposition \ref{pro1}, which suggests that the proposed DTA is disturbance-resilient and WGA cannot recover from disturbance on states.
To find why WGA cannot converge to the optimal solution, we obtain the following proposition.
\begin{prop}\label{pro1}
%Let $s=\max\{1-\beta K_1, \beta K_2-1\}$, 
%where $K_1=\frac{\underline{\eta}(1-\lambda_2)\left[\overline{\varphi}+(n-1)\underline{\eta}\right]}{\sqrt{n\left[\overline{\varphi}^2+(n-1)\underline{\eta}^2\right]}}$ and $K_2=\overline{\varphi}(1-\lambda_n)$.
%Suppose that $C_i(x), \forall i$ is strongly convex and Lipschitz continuous. 
Considering the following WGA,
\begin{align}
x_{i}(k+1)=&x_{i}(k)+\zeta_{i}(k)\notag
\\&-\alpha\sum_{i=1}^{n}w_{ij}(\nabla f_{i}(x_{i}(k))-\nabla f_{j}(x_{j}(k))),\label{gc2}
\end{align}
where the initial value $x_i(0), \forall i$ satisfies $\sum_{i=1}^{n}x_i(0)=\sum_{i=1}^{n}d_i$ and $\zeta_i(k)$ is the disturbance added to $x_i(k)$. $\bm{W}=\{w_{ij}\}_{n\times n}$ is doubly stochastic. Then, at least one of the following does not hold.

(i) $x_i(k), \forall i$ converges to a feasible $x_i(\infty), \forall i$, which satisfies $\sum_{i=1}^{n}x_i(\infty)=\sum_{i=1}^{n}d_i$;

(ii)$\sum_{k=0}^{\infty}\sum_{i=1}^{n}\zeta_i(k)\neq0$.
\end{prop}
The proof of Proposition 2 is shown in Appendix E.
Proposition 2 shows that $\sum_{k=0}^{\infty}\sum_{i=1}^{n}\zeta_i(k)=0$ is a necessary condition for (i). 
This implies that the effect from $\zeta_i(k)$ is accumulative and the perturbation of even only one agent at one iteration causes a deviation of convergence. This is the reason why WGA cannot converge to the optimal solution under disturbance. This Proposition is also applicable to DuSPA because $\bm{y}(k)$ in \cite{36} is also required to be feasible through all iterations. Any disturbance injected to $\bm{y}(k)$ leads to infeasible states. Figure 7 shows that the gaps between states and the optimal solutions caused by disturbance cannot be eliminated by WGA and DuSPA, which validates Proposition \ref{pro1}.
Compared with WGA and DuSPA, DTA is disturbance-resilient as shown in Theorem 2.

\section{Conclusion}

In this paper, we studied the distributed resource allocation problem over stochastic networks subject to random communication failure and disturbance on state. Based on the conditions of optimal solution, we proposed a disturbance-resilient distributed resource allocation algorithm by using deviation-tracking methods. Different from most existing algorithms that are only suitable for fixed networks, the proposed algorithm applies to stochastic networks.
The algorithm was proved to converge linearly to the optimal solution in mean square even with communication failure. Moreover, we find that the proposed algorithm has resilience to disturbance, i.e., our algorithm still converges to the optimal solution in mean square even under the disturbance on states. 
%we proved our algorithm still converges to the optimal solution in mean square. 
Based on the convergence result, we provided a method to improve the convergence rate. It was also proved that the proposed algorithm converges to the optimal solution in mean square even with uncoordinated stepsizes.
Future works will focus on constrained distributed resource allocation and application of the algorithm to smart grid.

\appendices
\begin{spacing}{1}

\section{Proof of Lemma 1}
%since for any $\bm{x}\in\mathbb{R}^n$,
%\begin{align}
%\frac{\left\Vert\bm{L}\bm{\Gamma}(\bm{I}-\bm{W})\bm{x}\right\Vert_2}{\left\Vert\bm{x}\right\Vert_2}&\leq\rho(\bm{L})\rho(\bm{\Gamma})\rho(\bm{I}-\bm{W})\notag
%\\&\leq\overline{\varphi}(1-\lambda_n),
%\end{align}
By definition, we obtain
\begin{align}
\frac{\Vert\bm{T}\bm{v}\Vert_2}{\Vert\bm{v}\Vert_2}\leq\rho(\bm{T})\leq\rho(\bm{L})\rho(\bm{\Gamma})\rho(\bm{I}-\bm{A})\leq\overline{\varphi}(1-{\lambda}_n(\bm{A})).\notag
\end{align}
This proves the right most inequality. 

Next, consider the left inequality of \eqref{lem1p1}. Due to Assumption 3,  we have $\bm{1}^T(\bm{I}-\bm{A})\bm{v}=0$. Let $\bm{z}=(\bm{I}-\bm{A})\bm{v}$, $z=[z_1,z_2,...,z_n]^T$, $z_i\in R, \forall i$. For that $\bm{1}^T\bm{v}=0$ and $\bm{v}\neq0$, we have $\bm{1}^T\bm{z}=0$ and $\bm{z}\neq0$, so there exists a full permutation of $\{1,...,n\}$, which is denoted by ${\iota_1,...,\iota_n}$, and $t\in\{1,...,n-1\}$ such that 
$z_{\iota_i}\geq0, \forall i\in\{0,...,t\}$, $z_{\iota_i}<0, \forall i\in\{t+1,...,n\}$ and 
$\sum_{i=1}^{t}z_{\iota_i}=-\sum_{i=t+1}^{n}z_{\iota_i}.$

Since $\bm{L}=\bm{I}-\frac{\bm{11}^T}{n}$ and $\frac{\bm{11}^T}{n}$ has only one nonzero eigenvalue $1$, whose corresponding eigenvectors lie in $span\{\bm{1}\}$, we consider the projection of $\bm{\Gamma}(\bm{I}-\bm{A})\bm{v}$ on $span\{\bm{1}\}$, i.e., $\frac{\bm{11}^T}{n}\bm{\Gamma}(\bm{I}-\bm{A})\bm{v}$.
We have $
\left\Vert\frac{\bm{11}^T}{n}\bm{\Gamma}(\bm{I}-\bm{A})\bm{v}\right\Vert_2=\frac{1}{\sqrt{n}}\bm{1}^T\bm{\Gamma}(\bm{I}-\bm{A})\bm{v}=\frac{1}{\sqrt{n}}\sum_{i=1}^{n}\gamma_{\iota_i}z_{\iota_i}.$
Then, 
\begin{align}
&\frac{\left\Vert\frac{\bm{11}^T}{n}\bm{\Gamma}(\bm{I}-\bm{A})\bm{v}\right\Vert_2^2}{\left\Vert\bm{\Gamma}(\bm{I}-\bm{A})\bm{v}\right\Vert_2^2}\leq\frac{\left(\sum_{i=1}^{n}\gamma_{\iota_i}z_{\iota_i}\right)^2}{n\sum_{i=1}^{n}\gamma_{\iota_i}^2z_{\iota_i}^2}\notag
\\&\leq\frac{\left(\sum_{i=1}^{t}\gamma_{\iota_i}z_{\iota_i}+\sum_{i=t+1}^{n}\gamma_{\iota_i}z_{\iota_i}\right)^2}{\frac{n}{t}(\sum_{i=1}^{t}\gamma_{\iota_i}z_{\iota_i})^2+\frac{n}{n-t}(\sum_{i=t+1}^{n}\gamma_{\iota_i}z_{\iota_i})^2}\notag
%\\&\leq\frac{1}{\sqrt{n}}\sqrt{\frac{\left(\sum_{i=1}^{t}\varphi_{\iota_i}z_{\iota_i}+\sum_{i=t+1}^{n}\eta_{\iota_i}z_{\iota_i}\right)^2}{\frac{(\sum_{i=1}^{t}\varphi_{\iota_i}z_{\iota_i})^2}{t}+\frac{(\sum_{i=t+1}^{n}\eta_{\iota_i}z_{\iota_i})^2}{n-t}}}\notag
\\&\leq\frac{\left(\overline{\varphi}-\underline{\eta}\right)^2\left(\sum_{i=1}^{t}z_{\iota_i}\right)^2}
{(\frac{n\overline{\varphi}^2}{t}+\frac{n\underline{\eta}^2}{n-t})(\sum_{i=t+1}^{n}z_{\iota_i})^2}
\leq\frac{(n-1)\left(\overline{\varphi}-\underline{\eta}\right)^2}{n\left[\overline{\varphi}^2+(n-1)\underline{\eta}^2\right]}.\label{lem12}
\end{align}
It thus follows that $
\frac{\left\Vert\bm{L}\bm{\Gamma}(\bm{I}-\bm{A})\bm{v}\right\Vert_2^2}{\left\Vert\bm{\Gamma}(\bm{I}-\bm{A})\bm{v}\right\Vert_2^2}
\geq\frac{\left[\overline{\varphi}+(n-1)\underline{\eta}\right]^2}{n\left[\overline{\varphi}^2+(n-1)\underline{\eta}^2\right]}.$
Then, we have
\begin{align}
\frac{\left\Vert\bm{L}\bm{\Gamma}(\bm{I}-\bm{A})\bm{v}\right\Vert_2^2}{\left\Vert\bm{v}\right\Vert_2^2}&=\frac{\left\Vert\bm{L}\bm{\Gamma}(\bm{I}-\bm{A})\bm{v}\right\Vert_2^2}{\left\Vert\bm{\Gamma}(\bm{I}-\bm{A})\bm{v}\right\Vert_2^2}\cdot\frac{\left\Vert\bm{\Gamma}(\bm{I}-\bm{A})\bm{v}\right\Vert_2^2}{\left\Vert\bm{v}\right\Vert_2^2}\notag
\\&\geq\frac{\underline{\eta}^2(1-\lambda_2(\bm{A}))^2\left[\overline{\varphi}+(n-1)\underline{\eta}\right]^2}{n\left[\overline{\varphi}^2+(n-1)\underline{\eta}^2\right]},\label{lem14}
\end{align}
which completes the proof the left part of \eqref{lem1p1}.

\section{Proof of Lemma 2}
We decompose $\Vert(\bm{I}-\beta\bm{T}_s)\bm{v}\Vert_2$ and obtain 
\begin{align}
&\mathbb{E}\{\Vert(\bm{I}-\beta\bm{T}_s)\bm{v}\Vert_2^2\}=\mathbb{E}\{\bm{v}^T(\bm{I}-\beta\bm{T}_s)^T(\bm{I}-\beta\bm{T}_s)\bm{v}\}\notag
%\\&=\mathbb{E}\{\bm{v}^T\mathbb{E}\{(\bm{I}-\beta\bm{T}_s)^T(\bm{I}-\beta\bm{T}_s)\}\bm{v}\}\notag
\\&=\mathbb{E}\{\bm{v}^T(\bm{I}+\beta^2\bm{T}_s^T\bm{T}_s-\beta\bm{T}_s^T-\beta\bm{T}_s)\bm{v}\}\notag
\\&=\mathbb{E}\{\Vert\bm{v}\Vert_2^2\}+\beta^2\mathbb{E}\{\Vert\bm{T}_s\bm{v}\Vert_2^2\}-2\beta\mathbb{E}\{\bm{v}^T\bm{T}_s\bm{v}\}.\label{lem1n0}
\end{align}
First, we consider $\mathbb{E}\{\Vert\bm{T}_s\bm{v}\Vert_2^2\}$. 
Since
\begin{align}
\frac{\Vert\bm{T}_s\bm{v}\Vert_2}{\Vert\bm{v}\Vert_2}\leq\rho(\bm{T}_s)\leq\rho(\bm{L})\rho(\bm{\Gamma})\rho(\bm{I}-\bm{W})\leq\overline{\varphi}(1-\underline{\lambda}_n),
\end{align}
We have $\mathbb{E}\{\Vert\bm{T}_s\bm{v}\Vert_2^2\}\leq\overline{\varphi}^2(1-\underline{\lambda}_n)^2\mathbb{E}\{\Vert\bm{v}\Vert_2^2\}$.

Next, we will find the lower bound of $\frac{\mathbb{E}\{\bm{v}^T\bm{T}_s\bm{v}\}}{\mathbb{E}\{\Vert\bm{v}\Vert_2^2\}}$.

Similar to \eqref{lem14}, we have 
\begin{align}
\frac{\Vert\bm{L}\bm{\Gamma}(\bm{I}-\mathbb{E}\{\bm{W}\})\bm{v}\Vert_2}{\Vert\bm{v}\Vert_2}\geq\frac{\underline{\eta}^2(1-\bar{\lambda}_{2e})^2\left[\overline{\varphi}+(n-1)\underline{\eta}\right]^2}{n\left[\overline{\varphi}^2+(n-1)\underline{\eta}^2\right]}.\notag
\end{align}

Therefore, for any vector $\bm{v}$ such that $\bm{1}^T\bm{v}=0$ and $\bm{v}\neq\bm{0}$,
\begin{align}
K_1\leq\frac{\Vert\bm{L}\bm{\Gamma}(\bm{I}-\mathbb{E}\{\bm{W}\})\bm{v}\Vert_2}{\Vert\bm{v}\Vert_2}\leq K_2.\label{lem1n4}
\end{align}

Because $\bm{L}$ and $\bm{I}-\mathbb{E}\{\bm{W}\}$ are positive semi-definite, $\bm{\Gamma}$ is positive definite, we get that the eigenvalues of $\bm{L}\bm{\Gamma}(\bm{I}-\mathbb{E}\{\bm{W}\})$ are nonnegative. Then, together with \eqref{lem14}, we have for any sufficiently small $\iota>0$,
\begin{align}
\Vert(\bm{I}-\iota\bm{L}\bm{\Gamma}(\bm{I}-\mathbb{E}\{\bm{W}\}))\bm{v}\Vert_2\leq(1-\iota K_1)\Vert\bm{v}\Vert_2.\label{lem1n1}
\end{align} 
On the other hand, 
\begin{align}
&\Vert(\bm{I}-\iota\bm{L}\bm{\Gamma}(\bm{I}-\mathbb{E}\{\bm{W}\}))\bm{v}\Vert_2\notag
\\&=\Vert\bm{v}\Vert_2^2+\iota^2\Vert\bm{L}\bm{\Gamma}(\bm{I}-\mathbb{E}\{\bm{W}\})\bm{v}\Vert_2^2-2\iota\bm{v}^T\bm{L}\bm{\Gamma}(\bm{I}-\mathbb{E}\{\bm{W}\})\bm{v}\notag
\\&\geq(1+\iota^2K_1^2)\Vert\bm{v}\Vert_2^2-2\iota\bm{v}^T\bm{L}\bm{\Gamma}(\bm{I}-\mathbb{E}\{\bm{W}\})\bm{v}.\label{lem1n2}
\end{align} 
Combining \eqref{lem1n1} with \eqref{lem1n2}, we have 
$
(1+\iota^2K_1^2)\Vert\bm{v}\Vert_2^2-2\iota\bm{v}^T\bm{L}\bm{\Gamma}(\bm{I}-\mathbb{E}\{\bm{W}\})\bm{v}\leq(1-\iota K_1)^2\Vert\bm{v}\Vert_2^2
$. 
Then, we obtain 
$
\bm{v}^T\bm{L}\bm{\Gamma}(\bm{I}-\mathbb{E}\{\bm{W}\})\bm{v}\geq K_1\Vert\bm{v}\Vert_2^2.
$
Since $\bm{L}$ and $\bm{\Gamma}$ is constant, we have
$
\mathbb{E}\{\bm{v}^T\bm{T}_s\bm{v}\}=\mathbb{E}\{\bm{v}^T\bm{L}\bm{\Gamma}(\bm{I}-\mathbb{E}\{\bm{W}\})\bm{v}\}\geq K_1\mathbb{E}\{\Vert\bm{v}\Vert_2^2\}.
$
Substituting it into \eqref{lem1n0}, we have 
\begin{align}
\mathbb{E}\{\Vert(\bm{I}-\beta\bm{T}_s)\bm{v}\Vert_2^2\}\leq&(1+\beta^2K_2^2-2\beta K_1)\mathbb{E}\{\Vert\bm{v}\Vert_2^2\}.
\end{align}
This completes the proof.

\section{Proof of Lemma 3}
Let $\mathcal{E}_1(k)=\{(i, j)|\{\mathbb{P}\{w_{ij}(k)>0\}>0\}\}$ and $\mathcal{E}_2(k)=\{(i, j)|\{\mathbb{P}\{w'_{ij}(k)>0\}>0\}\}$, 
where $w'_{ij}(k)=\sum_{t=1}^{n}w_{it}(k)w_{tj}(k)$. Since $\mathbb{P}\{w_{ij}(k)\geq0\}=1, \forall i, j$ and $\mathbb{P}\{w_{ii}(k)>0\}=1, \forall i$, we have that $\forall i, j$,
\begin{align}
\mathbb{P}\{w'_{ij}(k)>0\}\geq\mathbb{P}\{w_{ii}(k)w_{ij}(k)>0\}=\mathbb{P}\{w_{ij}(k)>0\},\notag
\end{align}
which suggests $\mathcal{E}_1(k)\subseteq\mathcal{E}_2(k)$. From Assumption \ref{network}, $\mathcal{E}_1(k)$ is connected, so, $\mathcal{E}_2(k)$ is also connected. Then, we can obtain \eqref{lemm3}. This completes the proof.

\section{Proof if Theorem 3}
Similar to the proof of Theorem 1, 
it is not difficult to obtain that there exist $m_0>0$ and $\max\{\vert1-\alpha\vert, q_{\zeta}\}<q_0<1$ such that 
$\Vert\bar{\bm{y}}(k)\Vert_2\leq w_{0}q_{0}^k$.
We can also obtain that 
\begin{align}
\nabla\check{\bm{f}}(\bm{x}(k+1))=&\left(\bm{I}-\beta\bm{L}\bm{\Theta}(k)(\bm{I}-\bm{W})\right)\nabla\check{\bm{f}}(\bm{x}(k))\notag
\\&-\alpha\bm{L}\bm{\Theta}(k)\check{\bm{y}}(k)-\alpha\bm{L}\bm{\Theta}(k)\bar{\bm{y}}(k)\notag
\\&+\bm{L}\bm{\Theta}(k)\bm{\zeta}(k).\label{mth1e5}
\end{align}
\begin{align}
\check{\bm{y}}(k+1)=&(\bm{W}-\alpha\bm{I})\check{\bm{y}}(k)+\check{\bm{\zeta}}(k)-\beta(\bm{I}-\bm{W})\nabla\check{\bm{f}}(\bm{x}(k)).
\label{mth1ex7}
\end{align}

%&\Vert\check{\bm{x}}(k+1)-\check{\bm{x}}(k)\Vert\leq m_1'q'^k\label{res11};

Next, we will prove 
$\Vert\check{\bm{y}}(k)\Vert_E\leq m_1'q'^k$ and $
\Vert\nabla\check{\bm{f}}(\bm{x}(k))\Vert_E\leq m_2'q'^k$
by mathematical induction based on \eqref{mth1e5} and \eqref{mth1ex7}.

It is not difficult to prove that when $k=0$, there exist $m_1', m_2'>0$ and $q'\in(q_0,1)$ such that \eqref{res21} and \eqref{res31} hold.

Then, we assume that for all $k\leq\kappa$, \eqref{res21} and \eqref{res31} hold.

When $k=\kappa+1$,
from \eqref{th1ex7}, we have 
\iffalse
\begin{align}
\Vert\check{\bm{x}}(k+1)-\check{\bm{x}}(k)\Vert_E&\leq\alpha\Vert\check{\bm{y}}(k)\Vert_E+\beta(1-\underline{\lambda}_{ne})\Vert\nabla\check{\bm{f}}(\bm{x}(k))\Vert_E\notag
\\&\leq \alpha m_1'q'^k+\beta(1-\underline{\lambda}_{ne})m_2'q'^k.\label{th1e8}
%\\\Vert\check{\bm{y}}(k+1)\Vert\leq\underline{\lambda}_{ne}\Vert\check{\bm{y}}(k)\Vert+\Vert\check{\bm{x}}(k+1)-\check{\bm{x}}(k)\Vert\notag
%\\\Vert\nabla\check{\bm{f}}(\bm{x}(k+1))\Vert\leq s\Vert\nabla\check{\bm{f}}(\bm{x}(k))\Vert+\alpha\eta_{\max}\Vert\check{\bm{y}}(k)\Vert
%\Vert\left(\bm{I}-\beta\bm{L}\Theta(k)(\bm{I}-\bm{W})\right)\Vert
\end{align}
\fi
%Based on \eqref{th1e6} and \eqref{th1e8}, we have 
\begin{align}
\Vert\check{\bm{y}}(k+1)\Vert_E\leq \left[s_1m_1'+\beta(1-\underline{\lambda}_{ne})m_2'+m_{\zeta}\right]q'^k.\label{th1ex8}
\end{align}
%where $s_1=\max\{\bar{\lambda}_{2e}-\alpha, \alpha-\underline{\lambda}_{ne}\}$.
From \eqref{th1e5} and \eqref{barcon}, we obtain
\begin{align}
\Vert\nabla\check{\bm{f}}(\bm{x}(k+1))\Vert_E
%\leq &s_2\Vert\nabla\check{\bm{f}}(\bm{x}(k))\Vert_E+\overline{\varphi}\Vert\bm{\zeta}(k)\Vert_E\notag
%\\&+\alpha\overline{\varphi}(\Vert\check{\bm{y}}(k)\Vert_E+\Vert\bar{\bm{y}}(k)\Vert_E)\notag
%\\\leq &s_2m_2'q'^k+\alpha\overline{\varphi}m_1'q'^k+\alpha\overline{\varphi}\Vert\bar{\bm{y}}(k)\Vert_E\notag
\leq \left[s_2m_2'+\alpha\overline{\varphi}(m_1'+m_0)+\overline{\varphi}m_{\zeta}\right]q'^k.\label{th1ex9}
\end{align}
%From Lemma 1, we know that $s_2=\max\{1-\beta K_1, \beta K_2-1\}$, where $K_1=\frac{\underline{\eta}(1-\bar{\lambda}_{2e})\left[\overline{\varphi}+(n-1)\underline{\eta}\right]}{\sqrt{n\left[\overline{\varphi}^2+(n-1)\underline{\eta}^2\right]}}$ and $K_2=\overline{\varphi}(1-\underline{\lambda}_{ne})$.
%where $s_2=\max_k\{\rho\left(\bm{I}-\beta\bm{L}\bm{\Theta}(k)(\bm{I}-\bm{W})\right)\}$. From Lemma 1, we know that $s>0$ and $s$ does not depend on $k$.
%where $q'\geq1-\alpha$.
Due to $\alpha<1-\underline{\lambda}_{ne}, \ \beta<\frac{2}{K_2}$, we have $s_1<1$ and $s_2<1$.
Since
%&\alpha<1-\underline{\lambda}_{ne}, \ \beta<\frac{2K_1}{K_2^2},
$\alpha\beta\overline{\varphi}(1-\underline{\lambda}_{ne})<(1-s_1)(1-s_2)$, similar to \eqref{mth1ex10} and \eqref{mth1ex11}, we have 
there exist $m_1'>0$ and $m_2'>0$ and $q'\in(q_0, 1)$ such that 
\iffalse
\begin{align}
&\alpha\beta\overline{\varphi}(1-\underline{\lambda}_{ne})(m_1'+m_0)+\beta(1-\underline{\lambda}_{ne})\overline{\varphi}m_{\zeta}\notag
\\&<\beta(1-\underline{\lambda}_{ne})(1-s_2)m_2'\notag
\\&< (1-s_1)(1-s_2)m_1'-(1-s_2)m_{\zeta}.\label{th1n1}
\end{align}
Then, we have there exists $q'\in(q_0, 1)$ such that
\fi
\begin{align}
\left[s_1m_1'+\beta(1-\underline{\lambda}_{ne})m_2'+m_{\zeta}\right]q'^k\leq m_1'q'^{k+1},\label{th1ex10}
\\\left[s_2m_2'+\alpha\overline{\varphi}(m_1'+m_0)+\overline{\varphi}m_{\zeta}\right]q'^k\leq m_2'q'^{k+1},\label{th1ex11}
\end{align}
Substituting \eqref{th1ex10} and \eqref{th1ex11} into \eqref{th1ex8} and \eqref{th1ex9} yields
\begin{align}
&\Vert\check{\bm{y}}(k+1)\Vert_E\leq m_1'q'^{k+1},\notag
\\&\Vert\nabla\check{\bm{f}}(\bm{x}(k+1))\Vert_E\leq m_2'q'^{k+1}.\notag
\end{align}
This completes the induction.
Then, we have 
\begin{align}
\Vert\bm{y}(k)\Vert_E\leq\Vert\bar{\bm{y}}(k)\Vert_E+\Vert\check{\bm{y}}(k)\Vert_E\leq(m_0+m_1')q'^k.
\end{align}
The proof is completed.
%and we obtain the range of $\alpha$, i.e.,
%\begin{align}
%\alpha<\frac{(1-s)(1-\underline{\lambda}_{ne})}{\beta\overline{\varphi}(1-\underline{\lambda}_{ne})+(1-s)},
%\end{align}
\section{Proof of Proposition 2}

We prove that if (ii) holds, (i) does not hold. 
Since $\bm{W}$ is doubly stochastic, we have 
$\sum_{i=1}^{n}\sum_{i=1}^{n}w_{ij}(\nabla f_{i}x_{i}(k))-\nabla f_{j}(x_{j}(k)))=0.$
We obtain from \eqref{gc2} that $
\sum_{i=1}^{n}x_{i}(k+1)=\sum_{i=1}^{n}x_{i}(k)+\sum_{i=1}^{n}\zeta_{i}(k).$
Then, we get$
\sum_{i=1}^{n}x_{i}(k+1)=\sum_{i=1}^{n}d_{i}+\sum_{k=0}^{\infty}\sum_{i=1}^{n}\zeta_i(k).$
We can see that if (ii) holds, $\sum_{i=1}^{n}x_{i}(\infty)\neq\sum_{i=1}^{n}d_{i}$, so (i) does not hold. This completes the proof.
\end{spacing}

\ifCLASSOPTIONcaptionsoff
  \newpage
\fi

\bibliographystyle{IEEEbib}
\bibliography{antidis}
\begin{IEEEbiography}[{\includegraphics[width=1in,height=1.25in,clip,keepaspectratio]{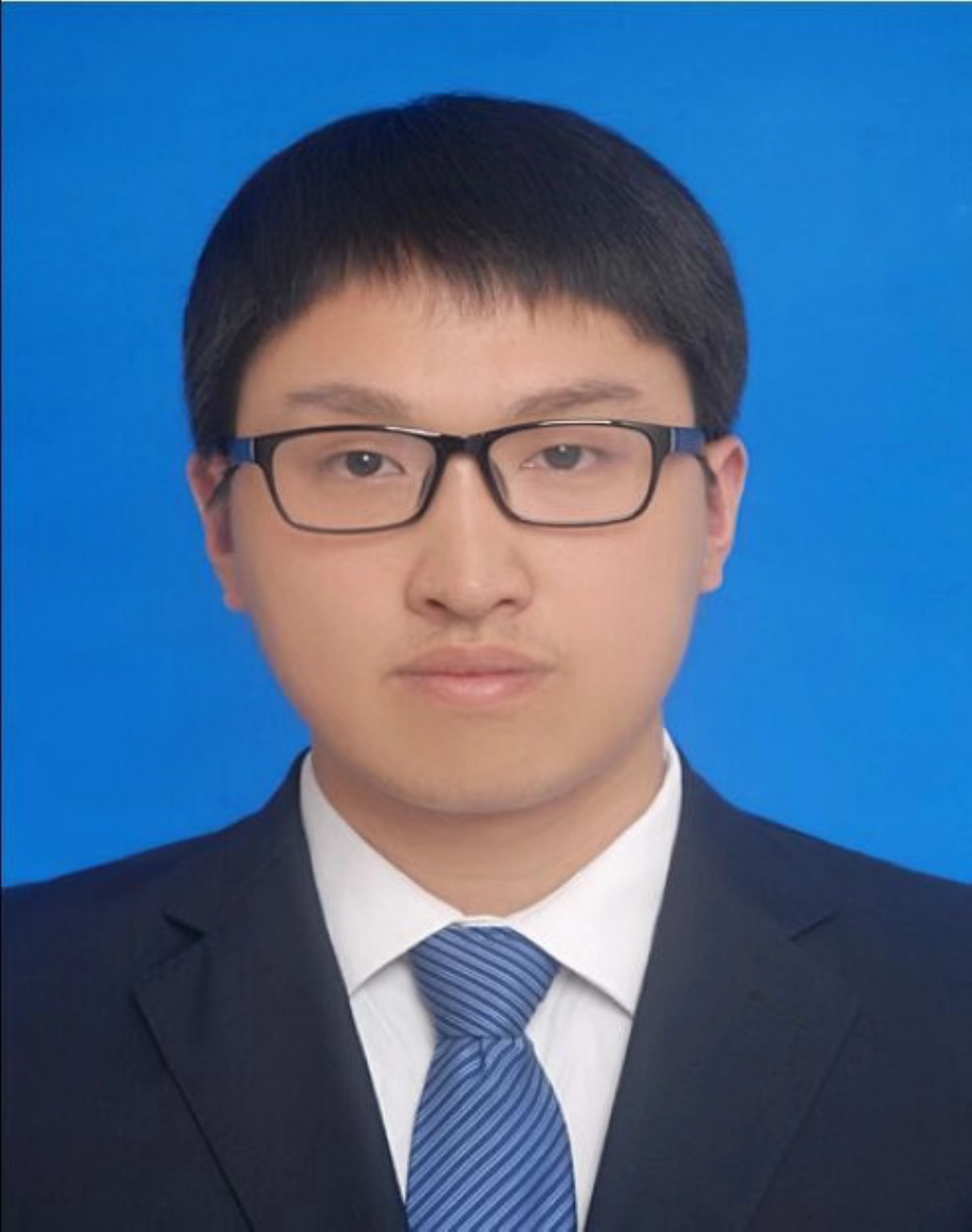}}]{Tie Ding}
received the B.Eng. degree in electrical engineering from Sichuan University, Chengdu, China, in 2016. He is currently pursuing the Ph.D. degree at the Department of Automation, School of Electronic Information
and Electrical Engineering, Shanghai Jiao Tong University, Shanghai, China.

His current research interests include distributed optimization and control in multi-agent systems,  and differential privacy in distributed networks.
\end{IEEEbiography}
\begin{IEEEbiography}[{\includegraphics[width=1in,height=1.25in,clip,keepaspectratio]{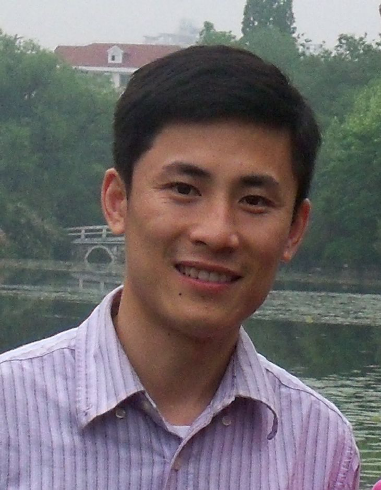}}]{Shanying Zhu}
(S'12-M'15) received the B.S. degree in information and computing science from the North China University of Water Resources and Electric Power, Zhengzhou, China, in 2006, the M.S. degree in applied mathematics from the
Huazhong University of Science and Technology, Wuhan, China, in 2008, and the Ph.D. degree in control theory and control engineering from Shanghai Jiao Tong University, Shanghai, China, in 2013.
From 2013 to 2015, he was a Research Fellow with the School of Electrical and Electronic Engineering, Nanyang
Technological University, Singapore, and also with the Berkeley Education Alliance for Research, Singapore. He joined Shanghai Jiao Tong University in 2015, where he is currently an Associate Professor with the Department
of Automation. 

His current research interests include multiagent systems and wireless sensor networks, particularly in coordination control of mobile robots and distributed detection and estimation in sensor networks, and their
applications in industrial networks.
%He is the recipient of China National Recruitment Program of 1000 Talented Young Scholars.
\end{IEEEbiography}
\begin{IEEEbiography}[{\includegraphics[width=1in,height=1.25in,clip,keepaspectratio]{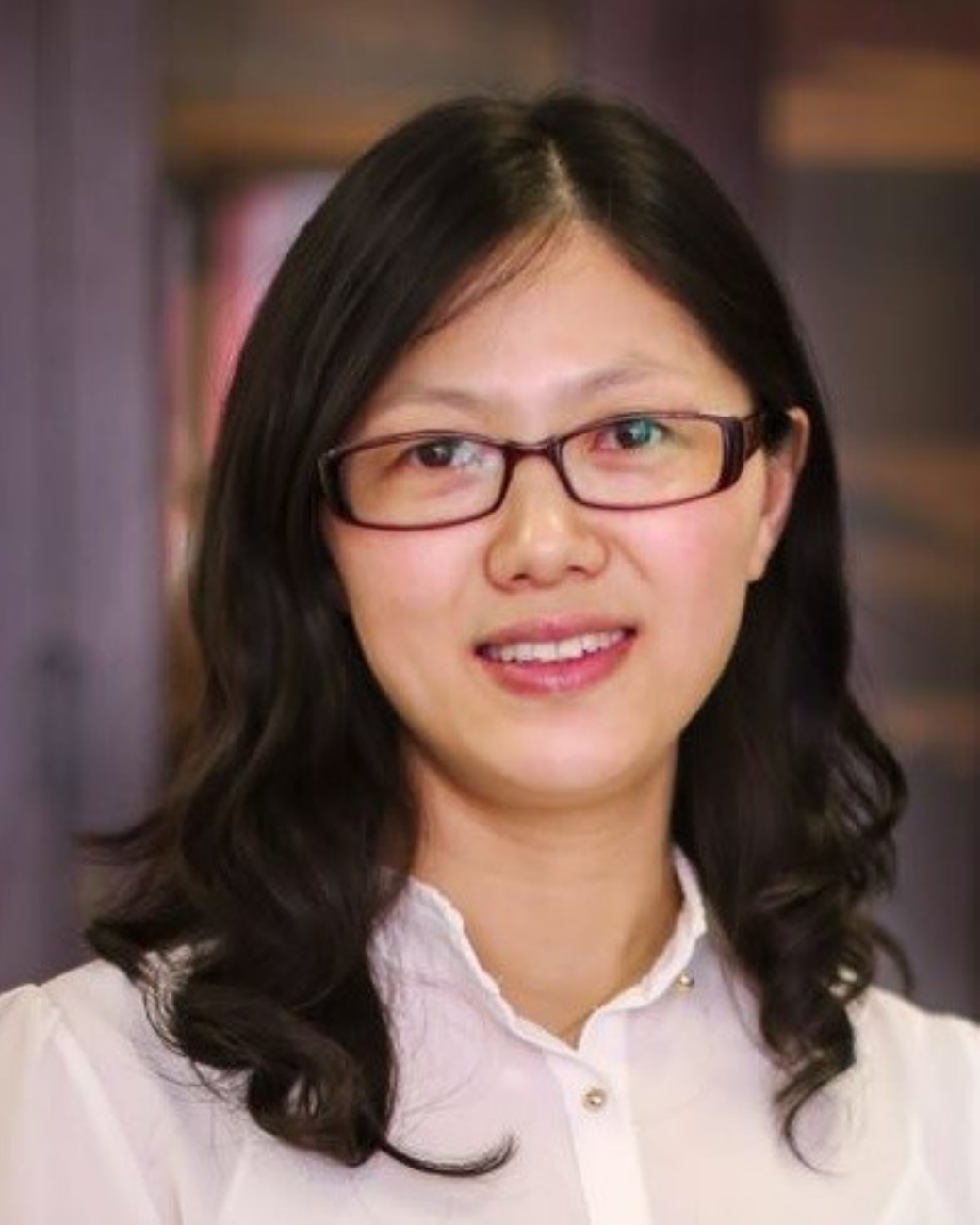}}]{Cailian Chen}
(S'03-M'06) received the B.Eng. and M.Eng. degrees in automatic control from Yanshan University, Qinhuangdao, China, in 2000 and 2002, respectively, and the Ph.D. degree in control and systems from the City University of Hong Kong, Hong Kong, in 2006. 

She was a Senior Research Associate with the City University of Hong Kong, in 2006 and a Post-Doctoral Research Associate with the University of Manchester, Manchester, U.K., from 2006 to 2008. She joined the Department of Automation, Shanghai Jiao Tong University, Shanghai, China, in 2008 as an Associate Professor, where she is currently a Full Professor. She was a Visiting Professor with the University of Waterloo, Waterloo, ON, Canada, from 2013 to 2014. She has authored or co-authored two research monographs and over 100 referred international journal and conference papers. She has invented over 20 patents. Her current research interests include wireless sensor networks and industrial applications, computational intelligence and distributed situation awareness, cognitive radio networks and system design, Internet of Vehicles and applications in intelligent transportation, and distributed optimization.

Dr. Chen was a recipient of the prestigious IEEE Transactions on Fuzzy Systems Outstanding Paper Award in 2008 and the Best Paper Award of the Ninth International Conference on Wireless Communications and Signal Processing in 2017. She was the recipient of the First Prize of Natural Science Award twice from the Ministry of Education of China in 2006 and 2016, respectively. She was honored as the ?Changjiang Young Scholar? by the Ministry of Education of China in 2015 and ?Excellent Young Researcher? by the NSF of China in 2016.
\end{IEEEbiography}

\begin{IEEEbiography}[{\includegraphics[width=1in,height=1.25in,clip,keepaspectratio]{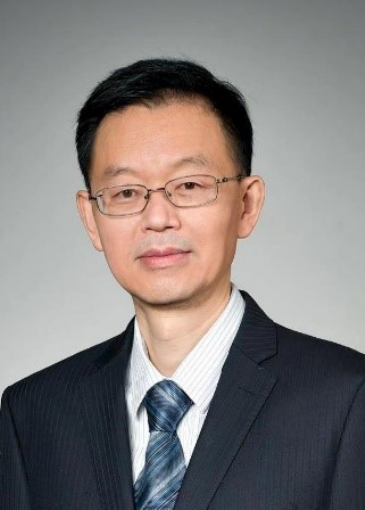}}]{Xinping Guan}
(M'04-SM'04-F'18) is currently a Chair Professor with Shanghai Jiao Tong University, Shanghai, China, where he is the Dean of the School of Electronic, Information and Electrical Engineering, and the Director of the Key Laboratory of Systems Control and Information Processing, Ministry of Education of China. He has authored
or co-authored 4 research monographs, over 270 papers in IEEE TRANSACTIONS and other peer-reviewed journals and numerous conference papers. His current research interests include industrial cyber-physical systems, wireless networking and applications in smart city and smart factory, and underwater sensor networks.

Dr. Guan was a recipient of the Second Prize of the National Natural Science Award of China in 2008, the First Prize of the Natural Science Award from the Ministry of Education of China in both 2006 and 2016, the First Prize of the Technological Invention Award of Shanghai Municipal, China, in 2017, the IEEE Transactions on Fuzzy Systems Outstanding Paper Award in 2008, the National Outstanding Youth Honored by the NSF of China, the Changjiang Scholar by the Ministry of Education of China, and the State-Level Scholar of New Century Bai Qianwan Talent Program of China. As a Principal Investigator, he has finished/been working on many national key projects. He is the Leader of the prestigious Innovative Research Team of the National Natural Science Foundation of China. He is an Executive Committee member of the Chinese Automation Association Council and the Chinese Artificial Intelligence Association Council.
\end{IEEEbiography}
% that's all folks
\end{document}